\documentclass[a4paper,10pt]{scrartcl}
\usepackage[utf8]{inputenc}
\usepackage[english]{babel}
\usepackage[fixlanguage]{babelbib}
\usepackage{cite}
\usepackage{amssymb, amsmath, amstext, amsopn, amsthm, amscd, amsxtra, amsfonts}
\usepackage{yfonts, mathrsfs}
\usepackage{paralist}

\usepackage[all,cmtip,line]{xy}
\CompileMatrices%

\usepackage[T1]{fontenc}
\usepackage{ifthen}
\usepackage{colonequals}

\usepackage[bookmarks,bookmarksnumbered,plainpages=false]{hyperref}

\swapnumbers
\newtheoremstyle{standard}
 {16pt}  
 {16pt}  
 {}  
 {}  
 {\bfseries}
 {}  
 { } 
 {{\thmname{#1~}}{\thmnumber{#2.}}\thmnote{~(#3)}} 

\newtheoremstyle{kursiv}
 {16pt}  
 {16pt}  
 {\itshape}  
 {}  
 {\bfseries}
 {}  
 { } 
 {{\thmname{#1~}}{\thmnumber{#2.}}\thmnote{~(#3)}} 

\theoremstyle{standard}
\newtheorem{defn} [subsection]{Definition}

\newtheorem{rem}   [subsection]{Remark}
\newtheorem{nota}   [subsection]{Notation}
\newtheorem{setup} [subsection]{}

\theoremstyle{definition}

\theoremstyle{kursiv}
\newtheorem{thm}[subsection]{Theorem}
\newtheorem{prop} [subsection]{Proposition}

\newtheorem{lem} [subsection]{Lemma}

\setdefaultenum{\normalfont (a)}{i.}{A.}{1.}

\renewcommand{\Re}{\mathrm{Re}}
\renewcommand{\Im}{\mathrm{Im}}

\newcommand{\Evol}{\mathrm{Evol}}
\newcommand{\im}{\mathrm{im}}

\newcommand{\evol}{\mathrm{evol}}
\newcommand{\ev}{\mathrm{ev}}

\newcommand{\id}{\mathrm{id}}

\newcommand{\N}{\mathbb{N}}

\newcommand{\R}{\mathbb{R}}
\newcommand{\K}{\mathbb{K}}
\newcommand{\C}{\mathbb{C}}
\newcommand{\SSS}{\mathbb{S}}
\newcommand{\SOne}{\SSS^1}

\newcommand{\V}{\mathcal{V}}

\newcommand{\Lf}{\ensuremath{\mathbf{L}}}

\renewcommand{\epsilon}{\varepsilon}

\newcommand{\vectw}[1]{\mathfrak{X}^\omega \left( #1 \right)}

\newcommand{\Fl}{Fl}
\newcommand{\set}[1]{\{  #1 \}}
\newcommand{\setm}[2]{\left\{\, #1 \middle\vert #2\,\right\}}
\newcommand{\norm}[1]{\left\lVert #1 \right\rVert}
\newcommand{\supnorm}[1]{\norm{#1}_\infty}
\newcommand{\abs}[1]{\left| #1 \right|}

\newcommand{\coloneq}{\colonequals}
\DeclareMathOperator{\res}{res}
\DeclareMathOperator{\pr}{pr}

\DeclareMathOperator*{\Diff}{Diff}
\DeclareMathOperator*{\Diffw}{\Diff\nolimits^{\omega}}

\DeclareMathOperator{\dom}{dom}
 	
\newcommand{\BHol}{\Hol_{\mathrm{b}}}
\newcommand{\Holb}{\BHol}
\newcommand{\Hol}{\mathrm{Hol}}
\newcommand{\A}{\Sigma}
\DeclareSymbolFont{bbold}{U}{bbold}{m}{n}
\DeclareSymbolFontAlphabet{\mathbbold}{bbold}

\newcommand{\nN}{\ensuremath{\mathcal{N}}}
\newcommand{\oO}{\ensuremath{\mathcal{O}}}

\newcommand{\func}[5]{#1 \colon #2 \rightarrow #3 , #4 \mapsto #5}
\newcommand{\smfunc}[3]{#1 \colon #2 \rightarrow #3}
\newcommand{\nnfunc}[4]{#1 \rightarrow #2 : #3 \mapsto #4}
\newcommand{\smset}[1]{ \left\{ #1 \right\} }
\newcommand{\oBallin}[3]{B_{#1}^{#2}\!\!\left(#3 \right)}
\newcommand{\cBallin}[3]{\overline{B}_{#1}^{#2}\left(#3 \right)}

\title{The Lie group of real analytic diffeomorphisms is not real analytic} 
\author{R. Dahmen\footnote{Technische Universität Darmstadt, Germany. 
\href{mailto:dahmen@mathematik.tu-darmstadt.de}{dahmen@mathematik.tu-darmstadt.de}}%
\ \  and A. Schmeding\footnote{NTNU Trondheim, Norway
\href{mailto:alexander.schmeding@math.ntnu.no}{alexander.schmeding@math.ntnu.no}
}%
}
\begin{document}

\maketitle

\begin{abstract}
We construct an infinite dimensional real analytic manifold structure for the space of real analytic mappings from a compact manifold to a locally convex manifold.
Here a map is real analytic if it extends to a holomorphic map on some neighbourhood of the complexification of its domain. 
As is well known the construction turns the group of real analytic diffeomorphisms into a smooth locally convex Lie group. 
We prove then that the diffeomorphism group is regular in the sense of Milnor.

In the inequivalent ``convenient setting of calculus'' the real analytic diffeomorphisms even form a real analytic Lie group. 
However, we prove that the Lie group structure on the group of real analytic diffeomorphisms is in general \emph{not} real analytic in our sense.
\end{abstract}

\medskip


\medskip


\tableofcontents

\section*{Introduction and statement of results}
A classical result by Leslie states that the group of analytic diffeomorphisms of a compact analytic manifold is a smooth (infinite dimensional) Lie group (see \cite{Leslie}). 
Unfortunately, the proof given in in the paper \cite{Leslie} contains a gap, as is pointed out in \cite{KM90}. 
However, in the ``convenient setting of analysis'' it is possible to achieve an even stronger result (cf.\ \cite{KM90,KM97}): 
\emph{The group of real analytic diffeomorphisms of a compact real analytic manifold is a regular real analytic (in the sense of convenient analysis) infinite dimensional Lie group}. 

In both approaches, the Lie group of real analytic diffeomorphisms has been modelled on Silva spaces. 
In the context of Silva spaces the concept of $C^r$-maps between locally convex spaces known as Keller's $C^r_c$-theory (see \cite{hg2002a} for a streamlined exposition) which is adopted in \cite{Leslie} and the ``convenient setting of analysis'' are equivalent. 
Hence, the result in \cite{KM90} subsumes the earlier (but flawed) result in \cite{Leslie}. \bigskip

In the present paper we consider again the Lie group of real analytic diffeomorphisms on a compact real analytic manifold.
The concept of real analyticity adopted in this work is based on an idea of Milnor. 
A map $f \colon E\supseteq U \rightarrow F$ between real locally convex spaces is called \emph{real analytic} if it extends to a holomorphic map (i.e.\ a Keller $C^\infty_\C$-map)
$\tilde{U} \rightarrow F_\C$ on an open neighbourhood $\tilde{U}$ of $U$ in the complexification $E_\C$ of $E$.

This notion of real analytic maps was introduced in \cite{hg2002a} and diverges from the real analytic maps in the convenient setting (see \cite{KM90,KM97}). 
It is known that both concepts coincide in the context of Fr\'{e}chet spaces. 
However, for Silva spaces real analyticity in our sense is stronger than ``convenient'' real analyticity.\bigskip

In the first part of the present paper, we construct a real analytic manifold structure on the set of real analytic mappings $C^\omega_\R (M,N)$. 
The line of thought follows the well known construction for smooth structures on spaces of (real) analytic mappings (cf.\ \cite{KM90}). 
We obtain the following result.\bigskip

\textbf{Theorem A} \emph{Let $M,N$ be real analytic manifolds.
Assume that $M$ is compact and that $N$ admits a local addition.
Then $C^\omega_\R (M,N)$ is a real analytic manifold.
The manifold structure does not depend on the choice of local addition.}\bigskip

Here the real analytic structure means real analyticity in the stronger sense explained above. 
To spell it out explicitly, this results is not surprising in any respect. 
In the inequivalent ``convenient setting of real analytic maps'' a version of Theorem A using the notion of ``convenient real analytic'' was already known. 
However, it is not possible to adapt the arguments establishing real analyticity to our setting, as the manifold is 
no longer modelled on a Silva space.
Moreover, the stronger notion of real analyticity we adopt in this paper requires special care if one wants to deal with spaces $C^\omega_\C (M,N)$ where $N$ is an infinite dimensional locally convex manifold.
To obtain Theorem A for infinite dimensional manifolds, we need to consider the complexification for certain types of infinite-dimensional vector bundles.
For finite dimensional vector bundles the concrete constructions seem to be part of the folklore (see for example \cite[7.1]{KM90}, an infinite dimensional analogue is now recorded in \cite{DGS2014}).

The manifold structure on $C^\omega_\R (M,N)$ allows us to adapt the well known construction in \cite[43.]{KM97} to construct the Lie group $\Diff^\omega (M)$. 
\bigskip

\textbf{Theorem (Kriegl/Michor 1990)} \emph{For a compact real analytic manifold $M$ the smooth structure of $C^\omega_\R (M,M)$ turns the group $\Diff^\omega (M)$ of real analytic diffeomorphisms into a Lie group. 
It is even a real analytic Lie group in the convenient sense.}\bigskip

Furthermore, in \cite[Theorem 43.3]{KM97} it has been shown that the the Lie group $\Diff^\omega (M)$ is a regular Lie group in the convenient setting.
To put this result into perspective, recall the notion of regularity for Lie groups: 

Let $G$ be a Lie group modelled on a locally convex space, with identity element $\mathbf{1}$, and
 $r\in \N_0\cup\{\infty\}$. We use the tangent map of the left translation
 $\lambda_g\colon G\to G$, $x\mapsto xg$ by $g\in G$ to define
 $v.g\coloneq T_{\mathbf{1}} \lambda_g(v) \in T_g G$ for $v\in T_{\mathbf{1}} (G) =: \Lf(G)$.
 Following \cite{dahmen2011} and \cite{1208.0715v3}, $G$ is called
 \emph{$C^r$-regular} if for each $C^r$-curve
 $\gamma\colon [0,1]\rightarrow \Lf(G)$ the initial value problem
 \begin{displaymath}
  \begin{cases}
   \eta'(t)&= \eta (t). \gamma(t)\\ \eta(0) &= \mathbf{1}
  \end{cases}
 \end{displaymath}
 has a (necessarily unique) $C^{r+1}$-solution
 $\Evol (\gamma)\coloneq\eta\colon [0,1]\rightarrow G$, and the map
 \begin{displaymath}
  \evol \colon C^r([0,1],\Lf(G))\rightarrow G,\quad \gamma\mapsto \Evol
  (\gamma)(1)
 \end{displaymath}
 is smooth. If $G$ is $C^r$-regular and $r\leq s$, then $G$ is also
 $C^s$-regular. A $C^\infty$-regular Lie group $G$ is called \emph{regular}
 \emph{(in the sense of Milnor}).
 Every finite dimensional Lie group is $C^0$-regular. Several
 important results in infinite-dimensional Lie theory are only available for
 regular Lie groups (see
 \cite{1208.0715v3}, cf.\ also \cite{KM97} and the references therein).\medskip

 In the present situation the model space of the Lie group $\Diff^\omega (M)$ is the Silva space $\vectw{M}$. 
 Thus our notion of differentiability is inequivalent to the notion of convenient differentiability on the space $C^\infty ([0,1],\vectw{M})$ and the result of Kriegl and Michor does not imply that the Lie group $\Diff^\omega (M)$ is regular in our setting. 
 However, Gl{\"o}ckner's regularity theorem for Silva Lie groups enables us to prove the regularity of $\Diff^\omega (M)$: 
 \bigskip

 \textbf{Theorem B} \emph{Let $M$ be a compact real analytic manifold. Then the Lie group $\Diff^\omega (M)$ is $C^1$-regular.}
 \bigskip

To say it once more, real analyticity in the sense adopted in this paper differs from ``convenient real analyticity''. 
In particular, the construction of the Lie group $\Diff^\omega (M)$ does not provide a real analytic Lie group structure in our setting albeit $\Diff^\omega (M)$ is a convenient real analytic Lie group. 
As the ambient manifold $C^\omega_\R (M,M)$ is a real analytic manifold in our sense the same is true for $\Diff^\omega (M)$.
Hence one would suspect that this structure turns $\Diff^\omega (M)$ into a real analytic Lie group.
In the last part of this paper we investigate this question for the group of analytic diffeomorphisms on the circle $\SOne$.
Contrary to the treatment in the ``convenient setting of analysis'', we obtain the following surprising result. \bigskip

\textbf{Theorem C} \emph{Let $\SOne$ be the unit circle in $\R^2$ with the canonical real analytic manifold structure. 
Then the group multiplication of the Lie group $\Diff^\omega (\SOne)$ is not real analytic.}  \bigskip

Thus $\Diff^\omega (\SOne)$ is a convenient real analytic Lie group (see \cite{KM90}) but not a real analytic Lie group in our sense.
The reason for this surprising behaviour is buried in the construction of the model space $\vectw{\SOne}$ of $\Diff^\omega (\SOne)$ and its complexification.
Both are Silva spaces, i.e.\ well behaved inductive limits of Banach spaces. 
Note that with some care, one can extend the group multiplication of $\Diff^\omega (M)$ to open zero-neighbourhoods of every individual steps of the inductive limit. 
However, we cannot choose an ascending sequence of zero-neighbourhoods leading to a zero-neighbourhood in the limit on which the multiplication is defined and complex analytic.

The counterexample in the second part of the paper is tailored to the manifold $\SOne$.
Hence it just indicates that $\Diff^\omega (M)$ will in general not be a real analytic Lie group in our sense.
Nevertheless, the construction of the counterexample should carry over to the general setting.
The authors believe that a similar analysis will show that for an arbitrary compact real analytic manifold $M$ (except the zero-dimensional ones) the group $\Diff^\omega (M)$ is not a real analytic Lie group.

\section{The locally convex manifold structure for spaces of analytic maps}\label{sect: mfd}

In this section we recall the construction of the manifold structure on spaces of analytic functions. 
The basic idea is not new and follows the exposition of the construction in the ``convenient setting'' (see \cite{KM90}). 
Beyond the Fr\'{e}chet setting our notion of real analytic maps is inequivalent to the notion in the convenient setting of global analysis. 
Thus the arguments establishing analyticity in our sense are new and require the complexification of several (infinite-dimensional) vector bundles.

\begin{nota}
 We write $\N \coloneq \set{1,2,\ldots}$, respectively $\N_0 \coloneq \N \cup \set{0}$. As usual $\K$ will denote either the field of real numbers $\R$ or the field of complex numbers $\C$, respectively.
 For a normed space $(E,\norm{\cdot})$, $x \in E$ and $R>0$, we let $B_R^E (x)$ be the open ball of radius $R$ around $x$. 
 \end{nota}
 
The setting of analytic mappings used in this paper was developed in \cite{hg2002a} building on a notion of complex analytic maps outlined in \cite{BS71b}.
For the concrete definition and more information on the differential calculus, locally convex manifolds and Lie groups we refer to Appendix \ref{app: appendix}.

\begin{defn}			\label{def: local addition}
 Let $N$ be a real analytic manifold modelled on a locally convex space over $\R$.
 We call a real analytic map $\A \colon TN \supseteq \Omega \rightarrow N$ defined on an open neighbourhood $\Omega$ of the zero-section in $TN$ a \emph{(real analytic) local addition} if 
  \begin{enumerate}
   \item $(\pi_{TN},\A) \colon TN \supseteq \Omega \rightarrow N \times N$ induces a $C^\omega_\R$-diffeomorphism onto an open neighbourhood of the diagonal in $N \times N$,
   \item $\A (0_x) = x, \ \forall x \in N$, where $0_x$ is zero-element in the fibre over $x$. 
  \end{enumerate}
\end{defn}

\begin{rem}
 \begin{enumerate}
  \item For finite dimensional paracompact manifolds $N$ there always exists a real analytic local addition.
  It is given by a real analytic Riemannian exponential map $\exp \colon TN\supseteq \Omega \rightarrow N$ (see \cite{MR0098847} and \cite[7.5]{KM90}).
  \item Note that also every (possibly infinite dimensional) (real analytic) Lie group admits a real analytic local addition 
  due to the real analytic group structure and the fact that the tangent bundle is trivial (cf.\ \cite[42.4]{KM97}).
  \end{enumerate}\label{rem: cov}
\end{rem}

Our approach uses complexifications of several (possibly infinite-dimensional) vector bundles on $M$. 
The reader is referred to \cite[Section 3]{DGS2014} for remarks on the notation and details concerning these constructions.

\begin{setup}\label{setup: realana}
 Let $(F,\pi,M)$ be a real analytic vector bundle whose typical fibre $E$ is a locally convex vector space, $M$ is finite dimensional paracompact and $F$ a locally convex manifold. 
 By \cite[Proposition 3.5]{DGS2014} (also cf.\ \cite[7.1]{KM90} for the finite dimensional case) the bundle $(F,\pi,M)$ admits a unique bundle complexification $(F_\C , \pi_\C, M_\C)$. \medskip
 
 For a compact subset $K \subseteq M$ consider the real vector \emph{space of germs along $K$ of real analytic sections} $\Gamma^\omega_\R (F|K)$.
 The complexification of $\Gamma^\omega_\R (F|K)$ is given as 
  \begin{displaymath}
    \Gamma^\omega_\R (F|K)_\C = \Gamma^\omega_\C (F_\C|K)
  \end{displaymath}
(cf.\ \cite[7.2]{KM90}).
As the bundle complexification is unique this construction does not depend on the choice of complexifications.

In the following we identify $\Gamma^\omega_\R (F|K)$ with the corresponding complemented subspace of $\Gamma^\omega_\C (F_\C|K)$ 
and topologise $\Gamma^\omega_\R (F|K)$ with the subspace topology. 
If $F$ is a finite dimensional manifold, Lemma \ref{lem: topchar} shows that $\Gamma^\omega_\C (F_\C|K)$ is a Silva space. 
Since closed subspaces of Silva spaces are Silva spaces by \cite[Corollary 8.6.9]{barr87} $\Gamma^\omega_\R (F|K)$ is a Silva space if $\dim F <\infty$.
\end{setup}

\begin{setup}[Canonical charts for $C^\omega_\R (M,N)$]\label{setup: charts}
 Fix a compact real analytic manifold $M$ and a (possibly infinite dimensional) real analytic manifold $N$. 
 We require that $N$ admits a real analytic local addition $\A \colon TN \supseteq \Omega \rightarrow N$.
 
 Consider $f \in C^\omega_\R (M,N)$ and define the subset 
  \begin{displaymath}
   U_f \coloneq \setm{g\in C^\omega_\R (M,N)}{ (f(x),g(x)) \in (\pi_{TN},\A)(\Omega), \text{ for all } x \in M}.
  \end{displaymath}
 of $C^\omega_\R (M.N)$ together with a map $\Phi_f \colon U_f \rightarrow \Gamma^\omega_\R (f^*TN)$ given by 
  \begin{displaymath}
   \Phi_f (\gamma) \coloneq (\id_M, (\pi_{TN} ,\A)^{-1} \circ  (f,\gamma)). 
  \end{displaymath}
 In the following, we identify $f^*\Omega = \{(x,X) \in M \times TN \mid X \in T_{f(x)} N \cap \Omega\}$ with an open submanifold of $f^*TN$. 
 The topology on $\Gamma^\omega_\R (f^*TN)$ is the subspace topology of $\Gamma^\omega_\C ((f^*TN)_\C|M)$. 
 Now $\Gamma^\omega_\C ((f^*TN)_\C|M)$ is the locally convex inductive limit of steps whose topology is the compact open topology (cf.\ Lemma \ref{lem: cotop}).
 We deduce that the topology of $\Gamma^\omega_\C ((f^*TN)_\C|M)$ is finer than the compact open topology and thus the same holds for $\Gamma^\omega_\R (f^*TN)$.
 As $M$ is compact, this shows that $\Phi_f (U_f) = \setm{\sigma \in \Gamma^\omega_\R (f^*TN)}{\sigma (M) \subseteq f^*\Omega}$ is an open subset of $\Gamma^\omega_\R (f^*TN)$. 
 Define the real analytic map 
   \begin{displaymath}
    \tau_f \colon f^*\Omega \rightarrow (f\times \id_N)^{-1} (\pi_{TN},\alpha)(\Omega) \subseteq M\times N,\tau_f (x,X) \coloneq (x,\A (X)). 
   \end{displaymath}
 Clearly $\tau_f$ is bijective and respect the fibres over $M$. Its inverse is the real analytic map
  \begin{displaymath}
   \tau_f^{-1} (y,z) \coloneq (y, (\pi_{TN},\A)^{-1} (f(y),z). 
  \end{displaymath}
 By construction this map takes its image in $f^*\Omega \subseteq f^*TN$ and is continuous with respect to the subspace topology on this space induced by $M\times TN$ on the fibre product.
 We conclude that $\tau_f$ is a homeomorphism onto its (open) image and thus $$\Omega_{f,g} \coloneq \tau_g^{-1} (\tau_f (f^*\Omega)) \subseteq g^*\Omega$$ is open. 
 Let us now compute for $f,g \in C^\omega_\R (M,N)$ and  $\sigma$ in $\Phi_g (U_f\cap U_g)$ a formula for $\Phi_f \circ \Phi_g^{-1}$. 
 Denote by $\pi_{TN}^*g \colon g^*TN \rightarrow TN$ the bundle map covering $g$. 
 Then we obtain
 \begin{equation}\label{eq: changes}\begin{aligned}
     (\Phi_f \circ \Phi_g^{-1} )(\sigma) &= (\id_M, (\pi_{TN},\A)^{-1} \circ (f, \A \circ (\pi_{TN}^*g)\circ \sigma)\\
                                         &= \tau_f^{-1} \circ \tau_g \circ \sigma \equalscolon (\tau_f^{-1} \circ \tau_g)_*  (\sigma).
                                         \end{aligned}
    \end{equation}
 Here $(\tau_f^{-1} \circ \tau_g)_* $ is defined on $\lfloor M, \Omega_{f,g}\rfloor \coloneq \setm{\sigma \in \Gamma^\omega_\R (g^*TN)}{\sigma (M) \subseteq \Omega_{f,g}}$
 which is an open subset of $\Gamma^\omega_\R (g^*TN)$.\footnote{As is customary, we denote for a smooth or analytic map $h$ by $h_*$ the map defined by $h_*(\gamma) \coloneq h \circ \gamma$ on a suitable open subset of a space of mappings.} 
 Thus $\Phi_g (U_f \cap U_g) = \lfloor M, \Omega_{f,g}\rfloor$ is open in the compact open topology and also in the finer topology on $\Gamma^\omega_\R(g^*TN)$.  
\end{setup}

 We will now prove Theorem A. 
 The manifold structure constructed on $C^\omega_\R (M,N)$ is a manifold under the extended Definition \ref{defn: ext} of manifolds (i.e. the model space can depend on the chart).
 \newpage 
\begin{thm}\label{thm: ram:mfd}
 Let $M,N$ be real analytic manifolds such that $M$ is compact and $N$ admits a real analytic local addition.
\begin{enumerate}
  \item The family $(U_f,\Phi_f)_{f \in C^\omega_\R (M,N)}$ is a real analytic atlas for $C^\omega_\R (M,N)$. 
  \item The identification topology with respect to the atlas in {\normalfont (a)} turns $C^\omega_\R (M,N)$ into a (Hausdorff) real analytic manifold modelled on the spaces $\Gamma^\omega_\R (f^*TN)$ where $f$ runs through $C^\omega_\R(M,N)$.
  If $N$ is finite dimensional, $C^\omega_\R (M,N)$ is a manifold modelled on Silva spaces.
  \item  The manifold structure does not depend on the choice of local addition. 
 \end{enumerate}
\end{thm}

\begin{proof}
 \begin{enumerate}
  \item Clearly $(U_f,\Phi_f)_{f \in C^\omega_\R (M,N)}$ is an atlas for $C^\omega_\R (M,N)$.
  We have to show that the changes of charts are real analytic. 
  To this end consider $f,g \in C^\omega_\R (M,N)$ with $U_f \cap U_g \neq \emptyset$ and fix $\sigma \in \Phi_g (U_f \cap U_g)$. 
  We will now construct a complex analytic map on an open $\sigma$-neighbourhood in the complexification $\Gamma^\omega_\R (g^*TN)_\C$ which extends $\Phi_f \circ \Phi_g^{-1}$. 
  To achieve this, proceed in several steps.\bigskip
  
 \textbf{Step 1: Complexifications of bundles and local data.}
 
 We form the pullback bundles $(f^*TN, f^*\pi_{TN},M)$ and $(g^*TN, g^*\pi_{TN}, M)$.
 These bundles are locally convex bundles over a finite dimensional paracompact base.
 Hence both pullback bundles admit unique bundle complexifications (cf.\ \ref{setup: realana}) which we denote by $((f^*TN)_\C, (f^*\pi_{TN})_\C,M^*)$ and $((g^*TN)_\C, (g^*\pi_{TN})_\C, M^*)$, respectively.
 Note that by passing to open subsets in the complexification, we may choose the same complexification $M^*$ of $M$ as base for the complex bundles.
 
 Following \ref{setup: realana} we identify the complexifications $\Gamma^\omega_\R (f^*TN)_\C = \Gamma^\omega_\C ((f^*TN)_\C|M)$ and $\Gamma^\omega_\R (g^*TN)_\C = \Gamma^\omega_\C ((g^*TN)_\C|M)$.
 Hence $\sigma \in \Phi_g (U_f \cap U_g) \subseteq \Gamma_\R^\omega (g^*TN)$ is associated to a germ $\tilde{\sigma} \in \Gamma^\omega_\C ((g^*TN)_\C|M)$ of a complex analytic section.

 Finally, we fix some local data.
 As $M$ is compact, we can choose a finite set $A$ with the following properties: 
  \begin{enumerate}
   \item[(i)] For $\alpha \in A$ there is a bundle trivialisations $\kappa_\alpha^f \colon (f^*\pi_{TN})^{-1} (M_\alpha) \rightarrow M_\alpha \times E_\C$ of $(f^*TN)_\C$ 
   which restricts on the real analytic submanifold $f^*TN$ to a bundle trivialisation of $f^*TN$.
   \item[(ii)] For $\alpha \in A$ there is a bundle trivialisations $\kappa_\alpha^g \colon (g^*\pi_{TN})^{-1} (M_\alpha) \rightarrow M_\alpha \times E_\C$ of $(g^*TN)_\C$ 
   which restricts on the real analytic submanifold $g^*TN$ to a bundle trivialisation of $g^*TN$.
   \item[(iii)] There are compact sets $K_\alpha \subseteq M_\alpha \cap M$ with $M = \bigcup_{\alpha \in A} K_\alpha$.
  \end{enumerate}
\bigskip

\textbf{Step 2: A suitable $\tilde{\sigma}$-neighbourhood $\oO_\sigma$ in $\Gamma^\omega_\R (g^*TN)_\C = \Gamma^\omega_\C ((g^*TN)_\C|M)$.}
 As $M$ is compact and thus finite dimensional, the complexification $M^*$ in Step 1 is paracompact and thus regular as a topological space.
 From Lemma \ref{lem: reg:bun} we deduce that thus $(f^*TN)_\C$ and $(g^*TN)_\C$ are regular as topological spaces.
 Observe that the image $\sigma (M)$ is a compact subset of $\Omega_{f,g} \subseteq g^*TN$.
 We thus deduce from \cite[Lemma 2.2 (a)]{DGS2014} that there is an open complex neighbourhood $O_1 \subseteq (g^*TN)_\C$ of $\sigma (M)$ on which
 $\tau_f^{-1} \circ \tau_g|_{\Omega_{f,g}} \colon \Omega_{f,g} \rightarrow f^*TN$ extends to a complex analytic map $\phi \colon (g^*TN)_\C \supseteq O_1 \rightarrow O_2 \subseteq (f^*TN)_\C$. 
 
 Without loss of generality we can assume $O_1 \cap f^*TN \subseteq \Omega_{f,g}$. 
 Shrinking $O_1$ further, we may assume that $\phi$ preserves the fibres of the complex bundle. 
 To see this note that $\tau_f^{-1} \circ \tau_g$ is fibre preserving. 
 Then we compute in pairs of trivialisations which satisfy (i) and (ii) of Step 1. It is easy to construct a complex analytic extension on a complex neighbourhood which preserves the fibres. 
 Finally the identity theorem for real analytic functions shows that we can shrink $O_1$ such that $\phi$ is a fibre preserving map of the complex bundle.
 
 As $\sigma (M)$ is contained in $O_1$, we deduce for each $\alpha \in A$ from (iii) in Step 1 that $\sigma (K_\alpha)$ is contained in the open set 
  \begin{displaymath}
   U_\alpha \coloneq \phi^{-1} ((f^*\pi_{TN})_\C^{-1} (M_\alpha)) \cap (g^*\pi_{TN})_\C^{-1} (M_\alpha) \cap O_1 \subseteq (g^*TN)_\C.
  \end{displaymath}
 Now fix $\alpha \in A$ and consider the family $(\kappa^g_\alpha (x))_{x \in K_\alpha}$. 
 Recall that $K_\alpha$ is a compact subset of a finite dimensional whence locally compact manifold. 
 Apply now Wallace theorem \cite[3.2.10]{Engelking1989} to obtain a finite family of compact sets $(K_{\alpha , k})_{1 \leq k \leq n_\alpha}$
 which satisfy the following properties. 
 \begin{itemize} 
  \item For $1\leq k \leq n_\alpha$ there are open subsets $O_{\alpha,k,1} \subseteq M_\alpha$ and $O_{\alpha,k,2} \subseteq E_\C$ such that $\kappa_\alpha^g (\sigma (K_{\alpha,k})) \subseteq O_{\alpha,k,1} \times O_{\alpha,k,2}$.
        This entails $K_{\alpha,k} \subseteq O_{\alpha,k,1} \subseteq M_\alpha$.
  \item The open set $O_{\alpha,k} \coloneq (\kappa^g_{\alpha})^{-1} (O_{\alpha, k,1} \times O_{\alpha, k,2})$ is contained in $U_\alpha$.
  \item $K_\alpha \subseteq \bigcup_{1\leq k \leq n_\alpha} K_{\alpha , k}$.
 \end{itemize}
 Repeat this construction for each $\alpha \in A$. 
 Then we can replace $(K_{\alpha})_{\alpha \in A}$ with a finite family of compact subsets such that $n_\alpha =1$ for all $\alpha \in A$ and the above conditions are satisfied.
 To shorten the notation denote this refinement again by $(K_{\alpha})_{\alpha \in A}$ and write $O_{\alpha,i} \coloneq O_{\alpha ,1,i}$ for $i \in \{1,2\}$ and $O_\alpha \coloneq O_{\alpha,1}$.
 
 Observe that by construction $\tilde{\sigma}$ is contained in each of the sets  
  \begin{displaymath}
   \lfloor K_\alpha , O_\alpha \rfloor \coloneq \setm{s \in \Gamma_\C^\omega ((g^*TN)_\C|M)}{s(K_\alpha) \subseteq O_{\alpha}}
  \end{displaymath}
The topology on $\Gamma_\C^\omega ((g^*TN)_\C)|M)$ is finer than the compact open topology (see Definition \ref{defn: ngermtop} and Lemma \ref{lem: cotop}). 
Thus the sets $ \lfloor K_\alpha , O_\alpha \rfloor$ are open.
Since $A$ is finite, the set $\oO_\sigma = \bigcap_{\alpha \in A} \lfloor K_\alpha , O_\alpha \rfloor$ is an open $\tilde{\sigma}$-neighbourhood.
Moreover, each $O_\alpha$ is contained in $O_1$ and thus property (iii) in Step 1 yields $\oO_\sigma \subseteq \lfloor M , O_1\rfloor = \{s \in \Gamma_\C^\omega ((g^*TN)_\C|M) \mid s(M)\subseteq O_1\}$.
The topology on $\Gamma_\C^\omega ((g^*TN)_\C|M)$ is finer than the compact open topology, whence the set $\lfloor M , O_1\rfloor$ is open in $\Gamma_\C^\omega ((g^*TN)_\C|M)$.
\bigskip

\textbf{Step 3: Embed the spaces of complex sections.} 
 With the notation of Lemma \ref{lem: topchar} we obtain topological embeddings 
  \begin{align*}
   \Theta_f &\colon \Gamma^\omega_\C ((f^*TN)_\C|M) \rightarrow \bigoplus_{\alpha \in A} \Hol (K_\alpha \subseteq M_\alpha^\C,E_\C), s \mapsto (I_{\kappa^f_{\alpha}} \circ \res^M_{K_\alpha} (s))_{\alpha \in A},\\
   \Theta_g &\colon \Gamma^\omega_\C ((g^*TN)_\C|M) \rightarrow \bigoplus_{\alpha \in A} \Hol (K_\alpha \subseteq M_\alpha^\C,E_\C), s \mapsto (I_{\kappa^g_{\alpha}} \circ \res^M_{K_\alpha} (s))_{\alpha \in A}.
  \end{align*}
 Observe that the topology on $\Hol (K_\alpha \subseteq M_\alpha,E_\C)$ is finer than the compact open topology 
 (Lemma \ref{lem: restro} (a) asserts $\Hol (K_\alpha \subseteq M_\alpha,E_\C) \cong \Gamma^\omega_\C ((f^*TN)_\C|K_\alpha)$ and the topology on the latter space satisfies this property).
 Hence, 
  \begin{displaymath}
   I_{\kappa^g_{\alpha}} \circ \res^M_{K_\alpha} (\lfloor K_\alpha , O_\alpha\rfloor) = \lfloor K_\alpha , O_\alpha^2\rfloor \subseteq \Hol (K_\alpha \subseteq M_{\alpha},E_\C)
  \end{displaymath}
 is open. 
 We deduce that $\Theta_g(\oO_\sigma)$ is contained in the open box-neighbourhood 
 $\bigoplus_{\alpha \in A} \lfloor K_\alpha, O_\alpha^2\rfloor \subseteq \bigoplus_{\alpha \in A} \Hol (K_\alpha \subseteq M_\alpha,E_\C)$.\bigskip
 
 \textbf{Step 4: A complex analytic extension on $\oO_\sigma$.}
 We define the map 
  \begin{displaymath}
   \phi_* \colon \Gamma_\C^\omega ((g^*TN)_\C|M)  \supseteq \lfloor M, O_1\rfloor \rightarrow \Gamma^\omega_\C ((f^*TN)_\C|M), s \mapsto \phi \circ s.
  \end{displaymath}
  Note that the map $\phi_*$ makes sense, as $\phi$ is fibre preserving. 
  Moreover, as $\oO_\sigma \subseteq \lfloor M, O_1\rfloor$ is a complex open neighbourhood of $\tilde{\sigma}$, 
  $\phi_*$ extends $(\tau_f^{-1} \circ \tau_g)_*$ in an open neighbourhood of $\sigma$ in the complexification.
    
  To see this define maps $f_\alpha \colon \lfloor K_\alpha , O_\alpha^2\rfloor \rightarrow \Hol (K_\alpha \subseteq M_\alpha^\C , E_\C)$ via 
  \begin{displaymath}
   f_\alpha (\gamma) \coloneq (\pr_2 \circ \kappa^f_{\alpha} \circ \phi \circ (\kappa^g_{\alpha})^{-1})_* (\id_{K_\alpha}, \gamma)
  \end{displaymath}
 By \cite[Proposition 3.3.]{hg2004b}, $f_\alpha$ is a $C^\infty_\C$-map and we obtain a map 
  \begin{displaymath}
   \oplus_{\alpha \in A} f_\alpha \colon  \oplus_{\alpha \in A} \Hol (K_\alpha \subseteq M^\C_\alpha, E_\C)\supseteq \oplus_{\alpha \in A} \lfloor K_\alpha , O_\alpha^2\rfloor \rightarrow \oplus_{\alpha \in A} \Hol (K_\alpha M^\C_\alpha, E_\C).
  \end{displaymath}
 This map satisfies $\oplus_{\alpha \in A} f_\alpha \circ \Theta_g|_{\oO_\sigma} = \Theta_f \circ \phi_*|_{\oO_\sigma}$. 
 Since every $f_\alpha$ is $C^\infty_\C$, the map $\oplus_{\alpha \in A} f_\alpha$ is $C^\infty_\C$ by \cite[Proposition 4.7]{hg2011}.
 Recall from Lemma \ref{lem: topchar} that $\Theta_f$ is a linear topological embedding with closed image. 
 Hence $\phi_*|_{\oO_\sigma} = \Theta_f^{-1} \circ \oplus_{\alpha \in A} f_\alpha \circ \Theta_g|_{\oO_\sigma}$ implies that $\phi_*|_{\oO_\sigma}$ is $C^\infty_\C$. 
 \bigskip
 
 We summarise now the results from the Steps 1-4. 
 We have seen that the map $(\tau_f^{-1}\circ \tau_g)_*$ extends to a complex analytic map on a neighbourhood of each element in its domain. 
 As real analyticity is a local property, this shows that $(\tau_f^{-1}\circ \tau_g)_*$ is real analytic.
 Hence \eqref{eq: changes} shows that $\Phi_f \circ \Phi_g^{-1}$ is a real analytic map. 
  
 \item Endow $C^\omega_\R (M,N)$ with the identification topology of the atlas constructed in (a). 
  The topological space $C^\omega_\R (M,N)$ will be a real analytic locally convex manifold if we can show that the identification topology on $C^\omega_\R (M,N)$ is Hausdorff. 
  To see this, it suffices to prove that for all $x \in M$ the point evaluations $\ev_x \colon C^\omega_\R (M,N) \rightarrow N, f \mapsto f(x)$  are continuous with respect to the identification topology. 
 By definition we thus have to prove that for all $f \in C^\omega_\R (M,N)$ and $x \in M$ the composition $\ev_x \circ \Phi_f^{-1} \colon \Phi_f (U_f) \rightarrow N$ is continuous.
 One easily computes for a section $\sigma \in \Phi_f (U_f)$ the identity $\ev_x \circ \Phi_f^{-1} (\sigma) = \A \circ \pi^*_{TN} (\sigma (x))$.
 Hence, it suffices to observe that the point evaluations on $\Gamma^\omega_\C ((f^*TN)_\C|M)$ are continuous (as $\Gamma^\omega_\R (f^*TN)$ is topologised as a subspace).
 By definition $\Gamma^\omega_\C ((f^*TN)_\C|M)$ is the inductive limit of spaces $\Gamma^\omega_\C ((f^*TN)_\C|W)$ on which the point evaluations are continuous (see Definition \ref{defn: ngermtop}). 
 As point evaluations are linear, an inductive limit argument shows that $\ev_x$ is continuous on the limit. 
 In conclusion $C^\omega_\R (M,N)$ is a Hausdorff topological space.
 \item  Let $\A^\#$ be another real analytic local addition on $N$. 
 Construct new charts $\Phi^\#_f$ and maps $\tau_f^\#$ as in \ref{setup: charts} with respect to $\A$. 
 We have only used that $(\pi_{TN}, \A)$ restricts to a diffeomorphism on an open neighbourhood $\Omega$ of the zero section in $TN$. 
 By definition of a local addition the same holds for $\A^\#$.
 Thus we can define $\tau_f^\# \circ \tau_g^{-1}$ on an open subset $\Omega_{f,g}^\# \subseteq g^* (\Omega \cap \Omega^\#)$ 
 (depending both on $\A$ and $\A^\#$). 
 Furthermore, we obtain an identity analogous to \eqref{eq: changes} (including now $\Phi_f^\#$ and $\tau_f^\#$). 
 Note that the arguments and constructions in Step 1, 3 and 4 of (a) do not depend on $\A$ and can thus be copied verbatim.
 Finally it is easy to see that also the construction in Step 2 of (a) can easily be adapted (as only the choices of compact and open sets have to be changed).
 Hence analogous arguments as in (a) show that the charts constructed with respect to different local additions are compatible, i.e.\ the resulting change of charts are real analytic.
 We conclude that the construction does not depend on the choice of local addition.\qedhere
 \end{enumerate}
\end{proof}

Recall that in \cite[\S 10]{Michor1980} a smooth manifold structure for $C^\infty_\R (M,N)$ has been constructed.
Moreover, the construction in \cite[Theorem 10.4]{Michor1980} carries over to an infinite-dimensional manifold $N$ which admits a local addition.\footnote{In fact, one has only to replace the $\Omega$-Lemma in the proof of \cite[Theorem 10.4]{Michor1980} by Gl\"{o}ckner's $\Omega$-Lemma \cite[Theorem F.23]{hg2004c}.}    
The manifold $C^\infty_\R (M,N)$ is modelled on spaces of smooth sections $\Gamma^\infty (f^*TN)$ for $f \in C^\infty_\R(M,N)$ with canonical charts defined analogously to the charts constructed in \ref{setup: charts}. 
We remark that the notion of differentiability adopted in \cite{Michor1980} coincides with the one adopted in this paper. 
Hence, the set $C^\infty_\R (M,N)$ is a smooth manifold and we can copy the arguments in \cite[p.48f.]{KM90} verbatim to obtain the following results: 

\begin{setup}[\hspace{-1pt}{\cite[Theorem 8.3]{KM90}}]\label{thm: cira}
\emph{Let $M$, $N$ be real analytic finite dimensional manifolds, with $M$ compact. 
Then the smooth manifold $C^\infty_\R (M,N)$ with the structure from \cite[10.4]{Michor1980} is a real analytic manifold.
In fact a real analytic atlas is given by
  \begin{displaymath}
   \Phi_f^\infty \colon C^\infty_\R (M,N) \supseteq U_f \rightarrow \Gamma^\infty (f^*TN), g \mapsto (\id_M, (\pi_{TN}, \A)^{-1} (f,g))
  \end{displaymath}
 where $f$ runs through $C^\omega_\R (M,N)$ and $U_f$ is defined as in \ref{setup: charts} with respect to the (real analytic) local addition $\A$.
 As $M$ is compact, the model spaces $\Gamma^\infty (f^*TN)$, are endowed with the compact open $C^\infty_\R$-topology.} 
\end{setup}

\begin{prop}\label{prop: cinftytop}
 Let $M$ and $N$ be real analytic manifolds and assume that $M$ is compact. 
 Consider the canonical inclusion $\iota \colon C^\omega_\R (M,N) \rightarrow C^\infty_\R (M,N)$. 
 \begin{itemize}
  \item [(a)] If $N$ is finite dimensional, then $\iota$ is a real analytic map with respect to the real analytic manifold structures of Theorem \ref{thm: ram:mfd} and \ref{thm: cira}. 
  \item [(b)] If $N$ is infinite-dimensional and admits a local addition, then $\iota$ is of class $C^\infty_\R$ with respect to the smooth structures of Theorem \ref{thm: ram:mfd} and \cite[\S 10]{Michor1980} on $C^\infty (M,N)$.
 \end{itemize}
\end{prop}

\begin{proof}
 Computing in canonical charts, we see that for $f \in C^\omega_\R (M,N)$ the canonical inclusion maps $\dom \Phi_f$ into $\dom \Phi_f^\infty$.
 Thus it suffices to consider the local representative $\Phi_f^\infty \circ \iota \circ \Phi_f^{-1}$ 
 which coincides with the restriction of the canonical inclusion $\Lambda \colon \Gamma^\omega_\R (f^*TN) \rightarrow \Gamma^\infty (f^*TN)$.
 As $\Lambda$ is linear it is thus sufficient to prove that $\Lambda$ is continuous.
 Denote the typical fibre of the vector bundle $f^*TN$ by $E$. 
 By definition of the compact open $C^\infty_\R$-topology, the topology on $\Gamma^\infty (f^*TN)$ is initial with respect to the linear mappings 
  \begin{displaymath}
   \theta_\psi^\infty \colon \Gamma^\infty (f^*TN) \rightarrow C^\infty_\R (M_\psi , E), X \mapsto \pr_2 \circ \psi \circ X|_{M_\psi}
  \end{displaymath}
 where $\psi$ runs through all (real analytic) bundle trivializations.

 Fixing a trivialization $\psi$ we will show that $\theta_\psi^\infty \circ \Lambda$ is continuous. 
 In the following we use standard multiindex notation to denote partial derivatives. 

 Recall from Definition \ref{defn: cocinfty} that a typical zero-neighbourhood $C^\infty_\R (M_\psi, E)$ is of the form
  \begin{displaymath}
   \Omega_{\kappa,K,n,p} \coloneq \setm{g\in C^\infty_\R (M_\psi, E)}{ \sup_{|\alpha| \leq n} P_{\alpha, \kappa , K,p} (g) = \sup_{|\alpha| \leq n}\sup_{x \in K} p\left(\partial^\alpha (g \circ \kappa^{-1})(x)\right) < 1 }, 
  \end{displaymath}
 where $\kappa \colon M_\psi \supseteq U_\kappa \rightarrow V_\kappa \subseteq \R^k$ is a real analytic manifold chart, $K\subseteq V_\kappa$ is compact, $ n \in \N_0$, and $p$ is a continuous seminorm on $E$.
 Fix $\kappa$, $K\subseteq V_\kappa$ compact, $n \in \N_0$ and a continuous seminorm $p$. 
 We construct a zero-neighbourhood in $\Gamma^\omega_\R (f^*TN)$ 
 which is mapped by $\theta^\infty_\psi \circ \Lambda$ to $\Omega_{\kappa,K,n,p}$. 
 
 The topology on $\Gamma^\omega_\R (f^*TN)$ is the subspace topology induced by $\Gamma^\omega_\C ((f^*TN)_\C|M)$.
 Since $\psi$ is a real analytic bundle trivialization, the complex analytic extension $\psi_\C$ of $\psi$ yields a bundle trivialization for $(f^*TN)_\C$ (cf.\ \cite[Proposition 3.5]{DGS2014}). 
 Analogously, we can extend $\kappa$ to a complex analytic chart $\kappa_\C$ of the complexification $M_\psi^\C$.
 Now by a combination of Lemma \ref{lem: restro} (b) and (a) we obtain a continuous linear map 
   \begin{displaymath}
    H \colon \Gamma^\omega_\C ((f^*TN)_\C|M) \rightarrow \Hol (K \subseteq M_{\psi_\C}, E_\C) , X \mapsto \pr_2 \circ \psi_\C \circ \res^M_K(X). 
   \end{displaymath}
 Recall that $ \Hol (K \subseteq M_{\psi_\C}, E_\C)$ is the inductive limit of the spaces $\Hol (U_k, E_\C)$ (where $(U_k)_{k \in \N}$ is a fundamental neighbourhood of $K$ in $M_{\psi_\C}$). 
 Each of the steps carries the compact open topology, which coincides with the compact open $C^\infty_\C$-topology for complex analytic maps by Lemma \ref{lem: cocinfty}.
 Hence by definition of the locally convex inductive limit, the set 
  \begin{displaymath}
   O_{\kappa_\C, K,n,p} \coloneq \setm{g \in \Hol (K \subseteq M_{\psi_\C}, E_\C)}{\sup_{|\alpha| \leq n} P_{\alpha,\kappa_\C, K, p_\C} (g) < 1} 
  \end{displaymath}
 is an open zero-neighbourhood where we chose a seminorm $p_\C$ on $E_\C$ which induces the given seminorm $p$ on the real subspace $E$. (This is possible by \cite[Section 2]{BS71a}.) 
 Note that the partial derivatives in the definition of $O_{K,n,p}$ are taken with respect to complex variables. 
 Since $\psi_\C|_{\dom \psi} = \psi$ and $\kappa_\C |_{M_\psi \cap U_{\kappa_\C}} = \kappa$, the zero-neighbourhood $H^{-1} (O_{\kappa_\C, K,n,p}) \cap \Gamma^\omega_\R (f^*TN)$ is mapped by $\theta_\psi^\infty \circ \Lambda$ into $\Omega_{\kappa,K,n,1}$.
\end{proof}

Keller $C^\infty_\K$-maps coincide on Silva spaces with smooth mappings in the convenient sense (cf.\ \cite[1.3]{KM90}). 
Thus we can almost\footnote{The authors were not able to follow the argument given in \cite[Theorem 43.3]{KM97} assuring that the identification topology on $C^\omega_\R (M,M)$ is finer than the compact open $C^\infty_\R$-topology.
Hence we chose to replace the argument showing that $\Diff^\omega (M)$ is an open subset in $C^\omega_\R (M,N)$.} 
copy the proof for the following result. 

\begin{prop}[{\hspace{-1pt}\cite[Theorem 43.3]{KM97}}]	\label{prop: diffomega}
 For a compact real analytic manifold $M$ the group $\Diffw (M)$ of all real analytic diffeomorphisms of $M$ is an open submanifold of $C^\omega_\R (M, M)$.
 
 Composition and inversion in this group are smooth, whence $\Diffw (M)$ is a smooth Lie group modelled on the Silva space $\vectw{M}$.
 Its Lie algebra is the space $\vectw{M}$ of all real analytic vector fields on $M$, equipped with the negative of the Lie bracket of vector fields. 
 The associated exponential mapping $\exp \colon \vectw{M} \rightarrow \Diffw (M)$ is the flow mapping to time $1$, and it is smooth.
\end{prop}

\begin{proof}
 Since $M$ is compact and thus finite dimensional,  Proposition \ref{prop: cinftytop} shows that the canonical inclusion $\iota \colon C^\omega_\R (M,M) \rightarrow C^\infty_\R (M,M)$ is real analytic. 
 Hence $\iota$ is continuous and we have $\Diff^\omega (M)= \iota^{-1} (\Diff^\infty (M))$.
 Now by \cite[Theorem 11.11]{Michor1980} $\Diff^\infty (M)$ is an open submanifold of $C^\infty_\R (M,M)$, whence $\Diff^\omega (M)$ is open in $C^\omega_\R (M,M)$.
 The rest of the proof can be copied verbatim from \cite[Theorem 43.4]{KM97}. 
\end{proof}

\begin{rem}\begin{enumerate}
            \item As shown in \cite[Theorem 43.3]{KM97} the Lie group $\Diffw (M)$ is even a real analytic Lie group in the convenient sense.
            Beyond the Fr\'{e}chet-setting (e.g.\ for Silva spaces) our notion of analyticity is inequivalent to the notion of convenient real analyticity.
            In Section \ref{sect: S1} it will turn out that the Lie group $\Diffw(M)$ is not real analytic in our sense. 
            \item The topology constructed on $\vectw{M} = \Gamma^\omega_\R (\id_M^* TM)$ coincides with the ``Van Howe''-topology constructed in \cite[\S 4]{Leslie} (this follows from \cite[Lemma 4.1]{Leslie}). 
            Thus the Lie group $\Diff^\omega (M)$ is modelled on the same topological vector space as the Lie group constructed in \cite{Leslie}. 
            Although the construction in \cite{Leslie} of the Lie group structure is flawed, the Lie group structure obtained in Proposition \ref{prop: diffomega} is precisely the one described in \cite{Leslie}.
           \end{enumerate} 
\end{rem}

\newpage
\section{Regularity of the group of analytic diffeomorphisms}

In this section we prove Theorem B, i.e.\ that $\Diff^\omega (M)$ is a regular Lie group.
This result is new in our setting (see \cite[Theorem 43.4]{KM97} for the corresponding result in the convenient setting) and the proof is based on H.\ Gl{\"o}ckner's regularity theorem for Silva Lie groups.

Throughout this section, let $M$ be a fixed compact $C^\omega_\R$-manifold.
Let us first study the differential equation for $C^k$-regularity of $\Diff^\omega (M)$. 

\begin{setup}[The differential equation of regularity for $\Diffw (M)$] \label{setup: reginDiff}
 Fix $\gamma \in C^k([0,1], \vectw{M})$. 
 If the evolution of $\gamma$ exists, it is a $C^{k+1}$-curve $\eta \colon [0,1] \rightarrow \Diff^\omega (M)$ which solves the differential equation
 \begin{equation}\label{eq: ode:reg}
  \begin{cases}
  \eta' (t) &= \gamma (t) .\eta (t) = T_1\rho_{\eta (t)} (\gamma (t)) \\
  \eta (0)  &= \id_M
  \end{cases}
 \end{equation}
 Recall that the right translation by an element in the group $\Diff^\omega (M)$ is precomposition with this element. 
 Identifying the tangent spaces $T_f \Diff^\omega (M) = \Gamma^\omega (f^* TM)$ we derive that $T_1 \rho_f (X) = X \circ f$ (this follows easily from a direct computation, using that the point evaluations on $\Gamma^\omega (f^* TM)$ are continuous linear and separate the points). 

 Hence the differential equation \eqref{eq: ode:reg} becomes
 \begin{equation}\label{eq: diff:Flow}
   \begin{cases}
    \eta' (t) &= \gamma (t,\eta (t)) \\
    \eta (0) &= \id_M
   \end{cases}
 \end{equation}
 We can view \eqref{eq: diff:Flow} as a differential equation whose right hand side is given by the time dependent vector field $\gamma$. 
 Since $M$ is compact, it follows from the theory of ordinary differential equations that the flow $\Fl^\gamma_0 \colon [0,1] \times M \rightarrow M$ of \eqref{eq: diff:Flow} is defined on $[0,1] \times M$. 
 Furthermore, for each fixed $t \in [0,1]$ the map $\Fl^\gamma_0 (t,\cdot) \in \Diff^\omega (M)$ and $\Fl_0^\gamma (0,\cdot) = \id_M$.\footnote{cf.\ \cite[IV. \S 2]{Lang95}. These results will also later be obtained as a consequence of the proof of regularity.} 
 We conclude that the evolution of $\gamma$ is $(\Fl^\gamma_0)^\vee (t) \coloneq \Fl_0^\gamma (t,\cdot) \in \Diff^\omega(M)$. 
 However, at this point it is not clear whether $(\Fl^\gamma_0)^\vee \colon [0,1] \rightarrow \Diff^\omega (M)$ is $C^{k+1}$.
 
 Thus $\Diff^\omega (M)$ will be a $C^k$-regular Lie group if we can show that $(\Fl^\gamma_0)^\vee$ is a $C^{k+1}$-curve and the evolution  
  \begin{displaymath}
   \evol \colon C^k ([0,1] ,\vectw{M}) \rightarrow \Diffw (M), X \mapsto (\Fl^X_0)^\vee (1) = \Fl^X_0 (1,\cdot).
  \end{displaymath}
  is smooth. Let us prepare the proof of these results with some auxiliary considerations.
\end{setup}

\begin{setup}
 As $M$ is a finite dimensional manifold, it admits a complexification $ M_\C$ which can be chosen to be a Stein manifold (see \cite[\S 3]{MR0098847}).
 The Stein manifold $ M_\C$ can be embedded as a closed complex submanifold of a finite dimensional complex space $\C^N$ by \cite[VII C.Theorem 13]{MR0180696}. 
 By \cite[Corollary 1]{Siu1976}, this complex submanifold admits an open neighbourhood $W\subseteq\C^N$ and an open holomorphic retract $\smfunc{ q}{W}{ M_\C}$.
 In the following, we will identify the tangent bundles $TM$ and $T M_\C$ as submanifolds of $T\C^N=\C^N\times\C^N$.
\end{setup}

\begin{setup}\label{setup: ext:diffeo}
 Recall that by \ref{rem: cov}(a), the real manifold $M$ admits a real-analytic local addition $\smfunc{\A}{\Omega_\A}{M}$, whence $\smfunc{(\pi_{TM},\A)}{\Omega_\A}{M\times M}$ is a $C^\omega_\R$-diffeomorphism onto its open image $(\pi_{TM},\A)(\Omega_\A)\subseteq  M\times M$. 
 Extend $(\pi_{T M},\A)^{-1}$ to a holomorphic map $\smfunc{h}{V_{\A}}{\C^N\times\C^N}$ on an open neighbourhood $V_{\A}\subseteq W\times W$ of the diagonal $\Delta_M$, i.e.~the following diagram commutes:
 \[
  \xymatrix{ 				 										
	\Delta_M \ar[r]^-\subseteq	&	(\pi_{TM},\A)(\Omega_\A) \ar[d]_\subseteq \ar[rr]_-{\cong}^-{(\pi_{T M},\A)^{-1}}	& & 		\Omega_\A	 \ar[r]^\subseteq \ar[d]_\subseteq & T M	\\
					&	V_{\A} 	  \ar[rr]^h				& &		\C^N\times\C^N &
}
 \]
 The set $V_{\A}\subseteq W\times W$ is an open neighbourhood of the compact set $\Delta_M$. Hence there is an $r>0$ such that $\Delta_M+ (\oBallin{r}{\C^N}{0}\times\oBallin{r}{\C^N}{0})\subseteq V_{\A}$.

 Here, $\oBallin{r}{\C^N}{0}$ denotes the open $r$-ball in $\C^N$ with respect to some norm. By replacing this norm by a multiple, we may assume in the following that $r=1$, i.e.
 \[
  \Delta_M + (\oBallin{1}{\C^N}{0}\times\oBallin{1}{\C^N}{0})\subseteq  V_{\A}.
 \]
 In particular, this implies that $ M+\oBallin{1}{\C^N}{0}\subseteq  W$.
\end{setup}

\begin{setup}\label{setup: sp:sets}
 We consider the following system of fundamental neighbourhoods of the compact set $M$ in the space $ W\subseteq\C^N$:
 \[
  U_{n}\coloneq M + \oBallin{1/n}{\C^N}{0}\subseteq  W\subseteq\C^N,
 \]
 together with the corresponding complex Banach spaces:
 \[
  F_{n}\coloneq\left( \BHol(U_{n},\C^N),\supnorm{\cdot} \right).
 \]
 Since $\smfunc{ q}{ W}{ M_\C}$ is open, the sets $ q(U_n)\subseteq  M_\C$ form a system of fundamental neighbourhoods of $ M$ in $ M_\C$.
 For $n\in\N$, we consider the \emph{real} Banach space
 \[
  E_n\coloneq\left( \setm{f\in \Holb(q(U_n) , \C^N)    }{f(a)\in T_a M \hbox{ for all } a \in  M}		,\supnorm{\cdot} \right)
 \]
 which can embeds via $\func{\theta_n}{E_n }{F_{n}}{g}{g\circ  q}$ isometrically into the complex Banach space $F_{n}$. We remark, that the maps $\theta_n$ induce isometric embeddings 
 \begin{displaymath}
  C([0,1], \theta_n) \colon C([0,1], E_n) \rightarrow C([0,1], F_{n}), \gamma \mapsto \theta_n \circ \gamma.
 \end{displaymath}
 Furthermore, there is a natural injective linear map $\nnfunc{E_{n}}{E_{n+1}}{f}{f|_{q(U_{n+1})}}$, which is a compact operator by \cite[Theorem 3.4]{KM90}.
 Hence, the direct limit of the sequence $(E_{n})_{n \in \N}$ is a Silva space which we identify with the Silva space $\vectw{M}$. 
 Using this identification, the limit maps of the inductive limit are 
 \[
  j_n \colon E_{n} \rightarrow \lim_{\to} E_{k} \cong \vectw{M}, f \mapsto (\id_M, f|_M).
 \]
\end{setup}

To prove the regularity of $\Diff^\omega (M)$ we exploit  Gl{\"o}ckners theorem for regularity of Silva Lie groups (\cite[Theorem 15.5]{1208.0715v3}) whose key point we repeat in the following Lemma.

\begin{lem}												\label{lem: star}
 The Lie group $\Diff^\omega (M)$ is $C^1$-regular if the following is satisfied:
 \begin{displaymath}
    (\star) \begin{cases} 
        \text{For each $n \in \N$ there is an open $0$-neighbourhood $P_n \subseteq C([0,1],E_n)$} \\ 
	\text{such that for $\gamma \in P_n$ the map $(\Fl^\gamma_0)^\vee \colon [0,1] \rightarrow \Diff^\omega (M), t \mapsto \Fl_0^\gamma (t,\cdot)$ is $C^1$}\\
	\text{and }\evol \colon P_n \rightarrow \Diff^\omega (M), \gamma \mapsto \Fl^\gamma_0 (1,\cdot) \text{ is continuous.}
        \end{cases}
 \end{displaymath}
\end{lem}

\begin{proof}
 Proposition \ref{prop: cinftytop} shows that the canonical inclusion $\Diff^\omega (M) \rightarrow \Diff (M)$ is a morphism of Lie groups which separates the points.
 Now $\Diff (M)$ is a $C^0$-regular Lie group by \cite[Corollary 13.7 (a)]{1208.0715v3} and $\Diff^\omega (M)$ is modelled on the Silva space $\vectw{M} = \lim_{\to} E_{n}$.
 Hence \cite[Theorem 15.5]{1208.0715v3} yields that $\Diff^\omega (M)$ will be $C^1$-regular if we can show that each element in $P_n$ admits a $C^1$-evolution and $\evol$ is continuous. 
 However, in \ref{setup: reginDiff} we have seen that for a (time dependent) real analytic vector field the flow solves the differential equation associated to regularity. 
 Thus if the flow induces a $C^1$-map $(\Fl^\gamma_0)^\vee$, this mapping is the evolution of $\gamma \in P_n$.
 We conclude that $\Diff^\omega (M)$ will be $C^1$-regular if the condition $(\star)$ is satisfied.
\end{proof}

Before we can establish $(\star)$ from Lemma \ref{lem: star}, we recall the definition of $C^{r,s}$-mappings from \cite{AS2014}.
\begin{defn}
 Let $E_1$, $E_2$ and $F$ be locally convex spaces, $U$ and $V$ open subsets of
 $E_1$ and $E_2$, respectively, and $r,s \in \N_0 \cup \{\infty\}$.
 \begin{enumerate}
  \item A mapping
 $f\colon U \times V \rightarrow F$ is called a $C^{r,s}$-map if for all
 $i,j \in \N_0$ such that $i \leq r, j \leq s$, the iterated directional
 derivative
 \begin{displaymath}
  d^{(i,j)}f(x,y,w_1,\dots,w_i,v_1,\dots,v_j) \coloneq (D_{(w_i,0)} \cdots
  D_{(w_1,0)}D_{(0,v_j)} \cdots D_{(0,v_1)}f ) (x,y)
 \end{displaymath}
 exists for all
 $ x \in U, y \in V, w_1, \ldots , w_i \in E_1,  v_1, \ldots ,v_j \in E_2$ and
 yields continuous maps
 \begin{align*} 
  d^{(i,j)}f\colon    U \times V \times E^i_1 \times E^j_2 &\rightarrow F,\\ 
  (x,y,w_1,\dots,w_i,v_1,\dots,v_j)&\mapsto (D_{(w_i,0)} \cdots D_{(w_1,0)}D_{(0,v_j)} \cdots D_{(0,v_1)}f ) (x,y).
 \end{align*}
 \item In (a) all spaces $E_1,E_2$ and $F$ were assumed to be modelled over the same $\K \in \{\R,\C\}$.
 By \cite[Remark 4.10]{AS2014} we can instead assume that $E_1$ is a locally convex space over $\R$ and $E_2,F$ are locally convex spaces over $\C$. 
 Then a map $f \colon U \rightarrow F$ is a $C^{r,s}_{\R,\C}$-map if the iterated differentials $d^{(i,j)}f$ (as in (a)) exist for all $0 \leq i \leq r , 0 \leq j \leq s$ and are continuous.
 Here the derivatives in the first component are taken with respect to $\R$ and in the second component with respect to $\C$.
 \end{enumerate}
\end{defn}

One can extend the definition of $C^{r,s}$- and $C^{r,s}_{\R,\C}$-maps to obtain $C^{r,s}$- or $C^{r,s}_{\R,\C}$-mappings on closed intervals (see \cite[Definition 3.2]{AS2014} for the general case of locally convex domains with dense interior).
For further results and details on the calculus of $C^{r,s}$-maps we refer to \cite{AS2014}.

With the help of the calculus of $C^{r,s}_{\R,\C}$-mappings at hand we can now establish condition $(\star)$ from Lemma \ref{lem: star}. 
To this end, let $n\in\N$ be fixed and set
\begin{displaymath}
 P_n \coloneq \oBallin{\frac{1}{4n}}{C([0,1],E_n)}{0} , \quad \quad  Q_n \coloneq \oBallin{\frac{1}{4n}}{C([0,1],F_{n} )}{0}.
\end{displaymath}

\begin{setup}\label{setup: fODE}
 Consider the map 
\begin{displaymath}
 \func{f}{[0,1] \times U_{n} \times Q_n}{\C^N}{(t,x,\gamma)}{\gamma (t) (x)}.
\end{displaymath}
We can rewrite $f$ as $f(t,x,\gamma) = \ev (\ev_1 (\gamma ,t) ,x)$ where $\ev_1 \colon C([0,1] , F_{n}) \times [0,1]\rightarrow F_{n}$ and $\ev \colon  \Holb(U_{n} , \C^N) \times U_{n} \rightarrow \C^N$ are the canonical evaluation maps. 
Note that the inclusion $ \Holb(U_{n} , \C^N) \rightarrow \Hol (U_{n} , \C^N)$ is continuous linear, whence holomorphic. 
We conclude from \cite[Proposition 3.20]{AS2014} that $\ev_1$ is a $C^{0,\infty}_{\R,\C}$-map and $\ev$ is holomorphic.
Thus the chain rule \cite[Lemma 3.17]{AS2014} implies that $f$ is of class $C^{0,\infty}_{\R,\C}$.
\end{setup}

Let us now consider the initial value problem 
 \begin{equation}\label{eq: DGL on UU4n}
  \begin{cases}
   x'(t) &	= 	f(t,x(t),\gamma), \\
   x(0)  &	=	x_0, \quad \text{for }x_0 \in U_{n} .
  \end{cases}
 \end{equation}
where the right hand side is given by $f$ from \ref{setup: fODE}.

\begin{lem}\label{lem: ODE:glob}
 Let $\gamma\in Q_n$. For every point $x_0\in U_{4n}$ the initial value problem (\ref{eq: DGL on UU4n})
 admits a unique maximal solution defined on the whole interval $[0,1]$. This solution takes its values in the set $U_{2n}$.
\end{lem}
\begin{proof}
 The map $f$ from \ref{setup: fODE} is of class $C^{0,\infty}_{\R,\C}$ with respect to $[0,1] \times \left( U_{n} \times Q_n \right)$. 
 Thus (\ref{eq: DGL on UU4n}) admits a unique maximal solution $\varphi_{0,x_0,\gamma}$ by \cite[Theorem 5.6]{AS2014}.
 We will now show that this solution takes its values in the compact set $\overline{U_{2n}}$ and hence will be globally defined, i.e.~ it is defined on $[0,1]$ (cf.\ \cite[IV \S 2 Theorem 2.3]{Lang95}). 
 To this end, let $t$ be an element in the domain of $\varphi_{0,x_0,\gamma}$. Then
 \begin{align*}
  \norm{ \varphi_{0,x_0,\gamma}(t) - x_0 }		& 	=			\norm{ \varphi_{0,x_0,\gamma}(t) - \varphi_{0,x_0,\gamma}(0)  }																					    		=			\norm{ \int_0^t \varphi_{0,x_0,\gamma}'(s)  \ ds	}	\\	
							& 	= 		\norm{ \int_0^t \gamma(s) \left(	\varphi_{0,x_0,\gamma}(s)	\right) \ ds	} 														\leq		\int_0^t \underbrace{\supnorm{  \gamma(s)}}_{<1/(4n)} \ ds			   												 	  < 		\frac{1}{4n}  .
 \end{align*}
 Since $x_0\in U_{4n}$, the triangle inequality implies that $\varphi_{0,x_0,\gamma}(t)\in U_{2n}\subseteq \overline{U_{2n}}$.
\end{proof}

\begin{prop}\label{prop: Evol:ex}
 For the time-dependent vector field $\gamma \in P_n$ the flow induces a $C^1$-map $(\Fl^\gamma_0)^\vee \colon [0,1] \rightarrow \Diff^\omega (M), t \mapsto \Fl^\gamma_0 (t,\cdot)$.
 Hence, $\Evol (\gamma) = (\Fl^\gamma_0)^\vee$ for $\gamma \in P_n$
\end{prop}

\begin{proof}
 Fix $\gamma \in P_n$.
 We use the holomorphic retract $q \colon  W \rightarrow  M_\C$ to lift $\gamma$ to an element $\tilde{\gamma} \coloneq C([0,1], \theta_n) (\gamma) \in Q_n$. 
 Note that by definition this map is given by $\tilde{\gamma} (t) (x) = \gamma (t) (q(x))$ for all $(t,x) \in [0,1] \times U_{n}$.
 
 Then the flow $\Fl^f_0 (\cdot,\tilde{\gamma}) \colon [0,1] \times U_{4n} \rightarrow U_{2n}$ of \eqref{eq: DGL on UU4n} exists for $\tilde{\gamma}$ by Lemma \ref{lem: ODE:glob} and yields a $C^{1,\infty}_{\R,\C}$-map by \cite[Proposition 5.9]{AS2014}. 
 The definition of the mapping $f$ shows that on $ M \subseteq U_{4n}$ we have 
  \begin{equation}\label{eq: Flow:prop}
   \Fl^f_0 (\cdot ,\tilde{\gamma})|_{[0,1] \times M} = \Fl^\gamma_0 \colon [0,1]\times  M \rightarrow  M.
  \end{equation}
 Recall now that there is a holomorphic extension $h \colon  V_{\A} \rightarrow \C^N \times \C^N$ of the diffeomorphism $(\pi_{T M} ,\A)^{-1}$ such that $U_{2n} \times U_{2n} \subseteq  V_{\A}$ (cf.\ \ref{setup: ext:diffeo} and \ref{setup: sp:sets}).
 Now the chain rules \cite[Lemma 3.17, Lemma 3.18]{AS2014} yield that 
  \begin{displaymath}
   F_\gamma \colon [0,1] \times U_{4n} \rightarrow \C^N , (t, x) \mapsto \pr_2 (x, \pr_2 \circ h \circ (x,\Fl^f_0 (t,x, \tilde{\gamma}))
  \end{displaymath}
 is a mapping of class $C^{1,\infty}_{\R,\C}$. 
 Moreover, since $h|_{ M \times  M} = (\pi_{T M} , \A)^{-1}$ this map satisfies 
  \begin{equation}\label{eq: tang:prop}
   F_\gamma ([0,1] \times \{a \}) \subseteq T_a M \subseteq \C^N \text{ for each } a \in  M. 
  \end{equation}
 Apply now the exponential law \cite[Theorem 3.28]{AS2014} for $C^{1,\infty}_{\R,\C}$-mappings on finite-dimensional manifolds to obtain a $C^1_\R$-map 
  \begin{displaymath}
   F_\gamma^\vee \colon [0,1] \rightarrow \Hol (U_{4n}, \C^N), t \mapsto F_\gamma(t,\cdot) .
  \end{displaymath}
 As $M \subseteq \overline{U_{6n}} \subseteq U_{4n}$ is compact, restriction induces a continuous linear inclusion $I_n \colon \Hol (U_{4n} , \C^N) \rightarrow \Holb (U_{6n},\C^N)$.
 Thus $I_n \circ F_\gamma^\vee \colon [0,1] \rightarrow \Holb (U_{6n},\C^N) = F_{6n}$ is a $C^1_\R$-mapping whose image is contained in the closed real subspace $E_{6n}$ by \eqref{eq: tang:prop}.
 Composing this map with the limit map $j_{6n} \colon E_{6n} \rightarrow \vectw{M}$ (see \ref{setup: sp:sets}) we finally obtain a $C^1_\R$ mapping 
  \begin{displaymath}
   H_\gamma \colon [0,1] \rightarrow \vectw{M} , t \mapsto F_\gamma(t,\cdot)|_{ M}^{T M} \stackrel{\eqref{eq: Flow:prop}}{=} (\pi_{T M},\A)^{-1} \circ (\id_M , \Fl^\gamma_0 (t, \cdot)). 
   \end{displaymath}
 Recall that $\Phi_{\id_M} \colon \Diff^\omega ( M) \supseteq U_{\id_M} \rightarrow \vectw{ M}, g \mapsto (\pi_{T M},\A)^{-1} \circ (\id_M , g)$ is a chart for $\Diff^\omega (M)$. 
 By construction $H_\gamma$ is contained in the image of $\Phi_{\id_M}$ and thus $\Phi_{\id_M}^{-1} \circ H_\gamma = (\Fl^\gamma_0)^\vee $ is a $C^1_\R$-map.
\end{proof}

By Proposition \ref{prop: Evol:ex} we know that for a fixed curve $\gamma\in Q_n$ the evolution $\Evol(\gamma)$ exists and is given by $(\Fl^\gamma_0)^\vee$. To apply Lemma \ref{lem: star}, it remains to show that the endpoint of $\Evol(\gamma)(1)$ depends continuously on $\gamma$:
\begin{prop}										\label{prop: evol continuous}
 The map $\func{\evol}{P_n}{\Diffw(M)}{\gamma}{(\Fl^\gamma_0)^\vee (1)}$ is continuous.
\end{prop}
\begin{proof}
 By Proposition \ref{prop: Evol:ex} the following map makes sense
 \[
  \func{H}{U_{4n} \times Q_n}{\C^N}{(a,\gamma)}{\pr_2 \left( h \left(	a, \Fl^f_0 (1,a,\gamma)	\right)          \right)}.
 \]
 We deduce from the proof of Proposition \ref{prop: Evol:ex} that the map $\smfunc{\evol}{P_n}{\Diffw(M)}$ can be written as $\evol = (\Phi_{\id_M})^{-1} \circ j_{6n}\circ\Psi_P$ where 
 $j_{6n}$ is the limit map (see \ref{setup: sp:sets}) and 
 \[
  \func{\Psi_P}{P_n}{E_{6n}}{\gamma}{\left(  a \mapsto H\big(a,C([0,1],\theta_n)(\gamma)\big)   \right).}
 \]
 Hence, continuity of $\evol$ follows as soon as we are able to show that $\smfunc{\Psi_P}{P_n}{E_{6n}}$ is continuous. To show this, consider the following commutative diagram:
 \[
  \xymatrix{ 	C([0,1],E_{n})	\ar[d]_{C([0,1],\theta_n)}		& &	P_n \ar[ll]_-{\supseteq} \ar[r]_{\Psi_P}\ar[d]		&		E_{6n}	\ar[d]^{\theta_{6n}}	\\
		C([0,1],F_{n})						& &	Q_n \ar[ll]_-{\supseteq} \ar[r]_{\Psi_Q}		&		F_{6n} 				
	    }
 \]
 with $\func{\Psi_Q}{Q_n}{F_{6n}}{\gamma}{\left(  a \mapsto H(a,\gamma)  \right).}$
 As $E_{6n}$ is isometrically embedded in $F_{6n}$, continuity of $\Psi_P$ will follow from continuity of $\Psi_Q$.
 To this end, we show that for fixed $\gamma_0\in Q_n$ the map $\Psi_Q$ is continuous at $\gamma_0$.

 By  \cite[Proposition 5.9]{AS2014} the flow $\smfunc{\Fl^f_0(1,\cdot)}{U_{4n}\times Q_n}{\C^N}$ is $C^\infty_\C$ and hence, the map $\smfunc{H}{U_{4n}\times  Q_n}{\C^N}$ defined above is $C^\infty_\C$ as well.
 It is well-known (see e.g.~\cite[Proposition 6.3]{Amann}) that this implies that the map $H$ is locally Lipschitz.
 We apply \cite[Proposition 6.4]{Amann} to the compact sets $\overline{U_{6n}}\subseteq U_{4n}$ and $\smset{\gamma_0}\subseteq Q_n$ to obtain an open $\gamma_0$-neighbourhood $W_{\gamma_0}$ such that $\smfunc{H|_{ \overline{U_{6n}} \times W_{\gamma_0} }}{ \overline{U_{6n}} \times W_{\gamma_0}}{\C^N}$
 is uniformly Lipschitz continuous with respect to $\gamma \in W_{\gamma_0}$, i.e.~there is a number $L_{\gamma_0}>0$ with $\nnfunc{W_{\gamma_0}}{\C^N}{\gamma}{H(a,\gamma)}$
 is $L_{\gamma_0}$-Lipschitz continuous for each $a\in\overline{U_{6n}}$.
 We conclude that the $F_{6n}$-valued map
 \[
  \func{\Psi_Q|_{W_{\gamma_0}}}{W_{\gamma_0}}{F_{6n}}{\gamma}{\left(  a \mapsto H(a,\gamma)  \right)}
 \]
 is $L_{\gamma_0}$-Lipschitz continuous. This finishes the proof.
\end{proof}

We can now summarise the results of this section as follows:

\begin{thm}
 The Lie group $\Diff^\omega (M)$ is $C^1$-regular.
\end{thm}

\begin{proof} 
 Proposition \ref{prop: Evol:ex} and Proposition \ref{prop: evol continuous} imply that condition $(\star)$ from Lemma \ref{lem: star} is satisfied. 
 Thus the assertion follows from that Lemma.
\end{proof}
\newpage

\section{The group of real analytic diffeomorphism on the circle is not real analytic}\label{sect: S1}
\newcommand{\CwS}{C^\omega_\R(\SOne,\SOne)}
\newcommand{\CwR}{C^\omega_\R(\SOne,\R)}
\newcommand{\CwC}{\Hol (\SOne\nobreak \subseteq\nobreak \C^\times,\nobreak \C)}
\newcommand{\ind}[1]{\mathbbold{1}_{#1}}

In this section, we will prove Theorem C. 
In particular, this implies that $\Diff^\omega (M)$ is not in general a real analytic Lie group.
To this end consider the unit circle $\SOne$ in $\C\cong\R^2$ with its canonical real analytic manifold structure. 
We begin with preparatory considerations concerning $\CwS$.

\begin{setup}[Real analytic local addition on $\SOne$]\label{setup: SONE}
The manifold $\SOne$ carries the structure of a real analytic (one-dimensional) Lie group and hence its tangent bundle is trivial via the following canonical isomorphism of real analytic vector bundles:
 \[
  \nnfunc{\SOne\times L(\SOne)}{T{\SOne}}{(z,v)}{(z,z \cdot v).}
 \]
 Since the Lie algebra  $L(\SOne)=T_1\SOne=i\R$ is isomorphic to $\R$, we obtain the following isomorphism:
 \[
  \func{\psi}{\SOne\times \R}{T{\SOne}}{(z,r)}{(z, z\cdot ir)}
 \]
The space $\vectw{\SOne}$ consists of all analytic sections in the tangent bundle $T{\SOne}$ and $\CwR$ can be viewed as analytic sections in the trivial bundle $\SOne\times\R$.
Hence, the spaces of sections are isomorphic as locally convex vector spaces. Note: The same argument also works if $\SOne$ is replaced by any other analytic Lie group.
For the rest of this section, we identify $T\SOne$ with $\SOne\times \R$ via the given isomorphism.\bigskip

The set $\Omega\coloneq\SOne\times ]\!-\!\pi,\pi[$ is an open neighbourhood of the zero-section in $T\SOne \cong \SOne\times\R$. 
 The manifold $\SOne$ admits a canonical real analytic local addition:
 \[
  \func{\Sigma}{\Omega}{\SOne}{(z,r)}{z\cdot e^{ir}.}
 \]
 In fact, the map
 \[
  \func{(\pi_{\SOne},\Sigma)  }{\Omega}{\setm{(z,w)\in\SOne\times\SOne}{z\neq -w }}{(z,r)}{(z,z\cdot e^{ir})}
 \]
 is an analytic diffeomorphism with inverse:
 \[
  \func{(\pi_{\SOne},\Sigma)^{-1}  }{\setm{(z,w)\in\SOne\times\SOne}{z\neq -w }}{\Omega}{(z,w)}{(z,\arg\left(\frac{w}{z}\right)}),
 \]
 where $\arg$ denotes the principal argument in the interval $\left]-\pi,\pi\right[$.
\end{setup}

As a next step, we consider the analytic manifold $\CwS$, constructed in Theorem \ref{thm: ram:mfd}.

\begin{setup}		\label{setup: mu}
 We want to consider a chart of the manifold $\CwS$ around the identity. 
 First, we observe that $\id_{\SOne}^*T\SOne = T\SOne\cong \SOne\times \R$.
 Thus, the canonical chart in \ref{setup: charts} around $\id_{\SOne}$ is given by:
 \begin{align*}
  \func{\Phi_{\id_{\SOne}}}     {& U_{\id_{\SOne}}}{V_{\id_{\SOne}}\subseteq \CwR}{\gamma}{\left(	z \mapsto \arg\left(\frac{\gamma(z)}{z}\right)	\right),}	\\
  \func{\Phi_{\id_{\SOne}}^{-1}}{& V_{\id_{\SOne}}}{U_{\id_{\SOne}}}{\eta}{\left(	z \mapsto z\cdot e^{i \eta(z)}\right).}
 \end{align*}
 Using this local chart around $\id_{\SOne}$, the composition map looks like
 \[
  \func{\mu}{\CwR\times \CwR}{\CwR}{(\eta_1,\eta_2)}{\eta_1 \circ E(\eta_2),}
 \]
 where $\func{E(\eta)}{\SOne}{\SOne}{z}{z \cdot e^{i\eta(z)}.}$
 We will now show that the map $\mu$ is \emph{not} real analytic in any open neighbourhood of $(0,0)$.
\end{setup}

\begin{setup}		\label{setup: chart_SOne}
 Viewing $\SOne$ as a subset of $\C^\times \coloneq \C \setminus \{0\}$, we may consider $\C^\times$ to be a complexification of $\SOne$. 
 This allows us to fix a fundamental sequence of neighbourhoods 
  \begin{displaymath}
   U_n \coloneq \setm{z \in \C}{e^{-\frac{1}{n}} <|z|< e^{ \frac{1}{n}}}
  \end{displaymath}
of $\SOne$ in its complexification.
 By \cite[4.2]{hg2004b}, the complexification of the locally convex space $\CwR$ is the Silva space $\CwC$.
 Now $\CwC$ is the locally convex direct limit of the following sequence of complex Banach spaces:
 \[
  \CwC = \bigcup_{n\in\N}\BHol(U_n,\C)
 \]
 To shorten the notation we set $E_n^b\coloneq \BHol(U_n,\C)$ using the fundamental sequence $(U_n)_{n \in \N}$ defined in \ref{setup: SONE} (b).
\end{setup}

\begin{prop}		\label{prop: mu not analytic}
 The map $\mu \colon \CwR\times \CwR \rightarrow \CwR$ introduced in \ref{setup: mu} is not real analytic on any neighbourhood of $(0,0) \in \CwR\times \CwR$.
\end{prop}

\begin{proof}
  \newcommand{\Umg}{\Omega}
  \newcommand{\CMult}{\mu_\C}
  \newcommand{\h}{f}
  \newcommand{\hh}{h}
  \newcommand{\hhh}{g}

 We argue by contradiction and assume that $\mu$ is real analytic in a neighbourhood of $(0,0)$.
 Then by definition of real analyticity,
 there exists an open $0$-neighbourhood $\Umg\subseteq \CwC$ and a complex analytic mapping $\smfunc{\CMult}{\Umg\times\Umg}{\CwC}$ such that $\CMult(\eta_1,\eta_2)\coloneq\mu(\eta_1,\eta_2)$ whenever $\eta_1,\eta_2\in\Umg\cap\CwR$.

 Since the linear map $\nnfunc{E_1^b}{\CwC}{f}{f|_{\SOne}}$ is continuous, 
 there is a number $R>0$ such that the closed ball $\cBallin{R}{E_1^b}{0}$ is mapped into the open neighbourhood $\Umg\subseteq \CwC$.

 Using this number $R$, we define the following meromorphic function:
 \[
  \func{\h}{\C^\times \setminus \smset{e^R,e^{-R}} }{\C}{z}{\frac{1}{z-e^R} + \frac{1}{\frac{1}{z}-e^R}}
 \]
 By construction, this function has the properties
 \[
  \h(\overline z) = \overline{\h(z)} \quad\hbox{ and }\quad \h(1/z)=\h(z) \quad \hbox{ for all } z.
 \]
 Combining these two properties, we may conclude that whenever $z\in\SOne$, we have $\h(z)\in\R$, since
 \[
  \overline{\h(z)} = \h(\overline z) = \h(1/z) = \h(z).
 \]
 The function $\h$ has poles of order $1$ at point $e^R$ and $e^{-R}$ (and a removable singularity at $0$ which is not important for our discussion) and is holomorphic elsewhere.
 
 As a next step, we fix a positive integer $n\in\N$ such that $\frac{1}{n}<R$ and find a number $\delta>0$ such that the closed ball $\cBallin{\delta}{E_n^b}{0}$ is mapped into the open neighbourhood $\Umg\subseteq \CwC$.

 Since the (relatively compact) open set $U_n$ has a positive distance from all the singularities of the function $\h$, the function $\h$ is bounded on the set $U_n$. Hence, $\h|_{U_n}\in E_n^b$. 
 We fix a scalar $r>0$ such that $r \cdot \h|_{U_n}\in\cBallin{\delta}{E_n^b}{0}$.
  
 For each complex number $z$, we denote by $z\cdot\ind{U_1}$ the constant function defined on $U_1$ taking the value $z$ at each point. For $\abs{z}\leq R$, the function $z\cdot\ind{U_1}$ belongs to $\cBallin{R}{E_1^b}{0}$ which is mapped into the set $\Umg$ when restricting the domain to $\SOne$.
 
 For each $z\in\cBallin{R}{\C}{0}$, we obtain that the pair 
  \begin{displaymath}
   (r \cdot \h|_{U_n}  , z \cdot\ind{U_1})\in	\cBallin{\delta}{E_n^b}{0}	 \times		\cBallin{R}{E_1^b}{0} \subseteq \Umg\times\Umg 
  \end{displaymath}
lies in the domain of the complex analytic map $\smfunc{\CMult}{\Umg\times\Umg}{\CwC}$. This allows us to define the following function
 \[
  \func{\hh}{\oBallin{R}{\C}{0} }{\C}{z}{ \CMult(r \cdot \h|_{U_n}  , z \cdot\ind{U_1}) (1). }
 \]
 As a composition of complex analytic mappings, this function is itself complex analytic, i.e.~holomorphic on the open disc $\oBallin{R}{\C}{0}\subseteq \C$.
 
 Now, let $z\in\oBallin{R}{\C}{0}\cap \R = \left] -R , R \right[$. Then, we can evaluate $\hh(z)$ explicitly:
 \begin{equation} \label{eq: ident}\begin{aligned}
  \hh(z)	  &		=		\CMult(r \cdot \h|_{U_n}  , z \cdot\ind{U_1}) (1)
  				=		\mu(r \cdot \h|_{U_n} , z \cdot\ind{U_1}) (1)
  				\\&		= 	r\cdot \h \circ E( z \cdot\ind{U_1})(1) 
  				= 	r\cdot \h \left( 1 \cdot e^{i z \cdot \ind{U_1}(1)}\right) 
  				\\&		= 	r\cdot \h \left( e^{i z}\right). 
  				\end{aligned}
 \end{equation}
 Note that for $|z|<R$, we have that
 \[
  |e^{i z}| = e^{\Re(i z)}  = e^{-i \Im(z)}
 \]
 and since $\Im(z)\in\left]-R,R\right[$, we can conclude that $e^{iz}$ is not one of the singularities of the function $\h$. 
 This shows that the holomorphic function 
 \[
  \func{\hhh}{\oBallin{R}{\C}{0} }{\C}{z}{r\cdot \h \left( e^{i z}\right)}
 \]
 makes sense and coincides with $\hh$ for all real arguments $z$ by \eqref{eq: ident}. 

 By the Identity Theorem for holomorphic functions, we obtain $\hh \equiv \hhh$ and thus the formula \eqref{eq: ident} holds for all $z \in \oBallin{R}{\C}{0}$.($z \in B_R^\C(0)$.)
 
 In particular, this allows us to conclude that $\hh(it) = r\cdot \h \left( e^{i (it)} \right)$ for all real $t\in\left]0,R\right[$ and hence
 \[
  \CMult(r \cdot \h|_{U_n}  , i t \cdot\ind{U_1}) (1) = r\cdot \h \left( e^{-t}\right).
 \]
 Now, we take the limit $t\to R$ on both sides and obtain a contradiction, since the left hand side converges to the perfectly well-defined number
 \[
  \CMult(r \cdot \h|_{U_n}  , iR  \cdot\ind{U_1}) (1)
 \]
 while the right hand diverges since $e^R$ is a pole of the function $\h$.
 \end{proof}
 
 We can now deduce the content of Theorem C.
 \begin{thm}
  The group multiplication $\Diff^\omega (\SOne)$ is not real analytic whence the Lie group $(\Diff^\omega (\SOne),\circ)$ is not a real analytic Lie group in our sense. 
 \end{thm}

 \begin{proof}
  By \ref{prop: diffomega}, the group $\Diff^\omega (\SOne)$ is an open neighbourhood of $\id_{\SOne}$ in the manifold $\CwS$. Pulling back the group multiplication by the canonical chart (\ref{setup: chart_SOne}), we obtain the mapping $\mu$ from \ref{setup: mu} on some open neighbourhood of $(0,0)$ in $\CwC\times \CwC$. Hence the assertion follows from Proposition \ref{prop: mu not analytic}.
 \end{proof}

\section*{Acknowledgements}
The research on this paper was partially supported by the project \emph{Topology in Norway} (Norwegian Research Council project 213458).
 
\begin{appendix}
\section{Locally convex calculus and spaces of germs of analytic mappings}\label{app: appendix}

In this appendix we recall several well known facts concerning calculus in locally convex spaces. 
Moreover, we discuss topologies on spaces of (germs) of analytic mappings. 
These results are well known, but it is sometimes difficult to extract the results and their proofs from the literature. 
Hence we repeat the results needed together with their proofs for the readers convenience.
 
\begin{defn}
 Let $r \in \N_0 \cup \set{\infty}$ and $E$, $F$ locally convex $\K$-vector spaces and $U \subseteq E$ open.
 We say a map $f \colon U \rightarrow F$ is a $C^r_\K$-map if it is continuous and the iterated directional derivatives 
  \begin{displaymath}
   d^kf (x,y_1,\ldots,y_k) \coloneq (D_{y_k} \cdots D_{y_1} f) (x)
  \end{displaymath}
 exist for all $k \in \N_{0}$ with $k \leq r$ and $y_1,\ldots,y_k \in E$ and $x \in U$,
 and the mappings $d^kf \colon U \times E^k \rightarrow F$ so obtained are continuous. 
 If $f$ is $C^\infty_\R$, we say that $f$ is \emph{smooth}.
 If $f$ is $C^\infty_\C$ we say that $f$ is \emph{complex analytic} or \emph{holomorphic} and that $f$ is of class $C^\omega_\C$.\footnote{Recall from \cite[Proposition 1.1.16]{dahmen2011}
 that $C^\infty_\C$ functions are locally given by series of continuous homogeneous polynomials (cf.\ \cite{BS71a,BS71b}).
 This justifies our abuse of notation.}     \label{defn: analyt}
\end{defn}

\begin{setup}[Complexification of a locally convex space]
 Let $E$ be a real locally convex topological vector space. We endow the locally convex product $E_\C \coloneq E \times E$ with the following operation 
 \begin{displaymath}
  (x+iy).(u,v) \coloneq (xu-yv, xv+yu) \quad \text{ for } x,y \in \R, u,v \in E
 \end{displaymath}
 The complex vector space $E_\C$ is called the \emph{complexification} of $E$. We identify $E$ with the closed real subspace $E\times \set{0}$ of $E_\C$. 
\end{setup}

\begin{defn}
 Let $E$, $F$ be real locally convex spaces and $f \colon U \rightarrow F$ defined on an open subset $U$. 
 We call $f$ \emph{real analytic} (or $C^\omega_\R$) if $f$ extends to a $C^\infty_\C$-map $\tilde{f}\colon \tilde{U} \rightarrow F_\C$ on an open neighbourhood $\tilde{U}$ of $U$ in the complexification $E_\C$.
\end{defn}

For $r \in \N_0 \cup \set{\infty, \omega}$, being of class $C^r_\K$ is a local condition, 
i.e.\ if $f|_{U_\alpha}$ is $C^r_\K$ for every member of an open cover $(U_\alpha)_{\alpha}$ of its domain, 
then $f$  is $C^r_\K$. (see \cite[pp. 51-52]{hg2002a} for the case of $C^\omega_\R$, the other cases are clear by definition.)  
In addition, the composition of $C^r_\K$-maps (if possible) is again a $C^r_\K$-map (cf. \cite[Propositions 2.7 and 2.9]{hg2002a}). 
 
\begin{setup}[$C^r_\K$-Manifolds and $C^r_\K$-mappings between them]
 For $r \in \N_0 \cup \set{\infty, \omega}$, manifolds modelled on a fixed locally convex space can be defined as usual. 
 The model space of a locally convex manifold and the manifold as a topological space will always be assumed to be Hausdorff spaces. 
 However, we will not assume that manifolds are second countable or paracompact.
 Direct products of locally convex manifolds, tangent spaces and tangent bundles as well as $C^r_\K$-maps between manifolds may be defined as in the finite dimensional setting.

For $C^r_\K$-manifolds $M,N$ we use the notation $C^r_\K(M,N)$ for the set of all $C^r_\K$-maps from $M$ to $N$.
Moreover, we let $\Diff^r_\K(M)$ denote the subset of all $C^r_\K$-diffeomorphisms in $C^r_\K (M,M)$. 
For $C^\infty_\C$-manifolds, we will also write $\Hol(M,N)\coloneq C^\omega_\C(M,N) \coloneq C^\infty_\C(M,N)$ for the set of all complex analytic maps from $M$ to $N$.

Furthermore, for $s \in \{\infty,\omega\}$, we define \emph{locally convex $C^s_\K$-Lie groups} as groups with a $C^s_\K$-manifold structure turning the group operations into $C^s_\K$-maps.
\end{setup}

To deal with manifolds of analytic mappings we need to slightly extend the notion of locally convex manifold. 
This is needed only in Section \ref{sect: mfd}, where we relax the definition of a manifold as follows.
 
\begin{defn}[Generalized manifolds]\label{defn: ext}
  Let $M$ be a Hausdorff topological space. 
    \begin{enumerate}
     \item A pair $(U_\kappa, \kappa)$ with $U_\kappa \subseteq M$ open and $\kappa \colon U_\kappa \rightarrow V_\kappa \subseteq E_\kappa$ a homeomorphism onto an open subset of a locally convex space $E_\kappa$ over $\K$ is called \emph{generalized manifold chart}.
      Note that the model space may change depending on on the chart.
     \item For $r \in \N_0 \cup \set{\infty,\omega}$ define $C^r_\K$-compatibility and $C^r_\K$-atlases for generalized manifold charts exactly as in the finite dimensional case.
          A generalized $C^r_\K$-manifold is a Hausdorff topological space with a $C^r_\K$-manifold structure, i.e. an equivalence class of $C^r_\K$-atlases induced by an atlas of generalized charts.  
    \end{enumerate}  
\end{defn}

Now we discuss the standard topologies on function spaces.


\begin{setup}[The Compact Open $C^\infty_\K$-Topology]				\label{defn: cocinfty}
 Let $M$ be a finite dimensional $C^\infty_\K$-manifold of dimension $d\in\N_0$ and let $E$ be any locally convex $\K$-vector space. The 
 \emph{compact open $C^\infty_\K$-topology} on the vector space $C^\infty_\K(M,E)$ is the locally convex vector spaces topology, given by the family of seminorms
    \[
     \func{P_{\alpha,\phi,K,p}}{ C^\infty_\K(M,E) }{\left[0,+\infty\right[}
     { \gamma } { \sup_{x\in K} p\left(	\partial^\alpha(\gamma\circ \phi^{-1})(x)	\right)},
    \] 
    where $p$ is a continuous seminorm on $E$,
    $\alpha\in\N_0^d$ is a multi-index,
    $\smfunc{\phi}{U_\phi}{V_\phi}$ is a $C^\infty_\K$-diffeomorphism of an open subset $U_\phi\subseteq M$ to an open subset $V_\phi\subseteq \K^d$, and
    $K\subseteq V_\phi$ is a compact set.
\end{setup}

In the case $\K=\C$, the space $C^\infty_\C(M,E)$ is endowed just with the compact open topology.
However, it is a well-known fact that the compact open topology coincides in this case with the topology from Definition \ref{defn: cocinfty}. 
For the reader's convenience we give a sketch of the proof.

\begin{lem}																																\label{lem: cocinfty}
 Let $M$ be a finite dimensional complex manifold and let $E$ be a complex locally convex vector space. Then the compact open $C^\infty_\C$-topology on the space
 \[
  \Hol(M,E)=C^\infty_\C(M,E),
 \]
 defined in \ref{defn: cocinfty} agrees with the usual compact open topology, which is the topology of uniform convergence on compact subsets.
\end{lem}
\begin{proof}
 Since $M$ is locally compact, we may work in local charts and hence, assume that $M=\Omega\subseteq\C^d$ is an open subset of a finite dimensional vector space $\C^d$.
 Let $a\in\Omega$ be a point. Then there is a number $R>0$ such that $K:=\cBallin{R}{\C^d}{a}\subseteq \Omega$. We fix the number $r:=R/2$.
 Let $x\in\oBallin{r}{\C^d}{a}$ and let $v\in\C^d$ be any vector of norm $1$. Then by Cauchy's integral formula, we may the write the derivative of $\gamma\in\Hol(\Omega,E)$ at point $a$ in direction $v$ as
 \[
  d^1\gamma(x,v)=\frac{1}{2\pi i} \int_{\abs{z}=r} \frac{\gamma(x+z v)}{z^{2}} dz.
 \]
 Applying a continuous seminorm $p$ on both sides, we obtain
 \[
  p\left(d^1\gamma(x,v)\right) \leq \frac{1}{2\pi}\cdot 2\pi r \sup_{\abs{z}=r} \frac{p(\gamma(x+zv))}{r^2}
  \leq \frac{1}{r} \sup_{y\in K} p(\gamma(y)). 
 \]
 In particular, if chose $v=e_j$ as a standard basis vector of $\C^d$, we obtain
 \[
  p\left(\frac{\partial\gamma}{\partial x_j}(x)\right) \leq \frac{1}{r} \sup_{y\in K} p(\gamma(y))
 \]
 and since $x\in\oBallin{r}{\C^d}{a}$ was arbitrary, this implies that uniform convergence on $K$ implies uniform convergence on the open ball $\oBallin{r}{\C^d}{a}$. Since every compact subset $K\subseteq\Omega$ can be covered by finitely many open balls, we have shown that taking the partial derivative $\frac{\partial }{\partial x_j}$ is continuous with respect to the compact open topology. The case of a derivative with respect to a multi-index follows by induction.
\end{proof}

\begin{setup}[The space of bounded holomorphic functions]
 For a finite dimensional complex manifold $M$ and complex Banach space $E$, the \emph{space of bounded holomorphic functions}
 \[
  \BHol(M,E)\coloneq \set{\gamma\in\Hol(M,E) \colon \gamma \text{ is bounded on }M },
 \]
 is a Banach space with respect to the supremum norm (cf.\ \cite[Proposition 6.5]{BS71b}).
\end{setup}


\begin{setup}[Fundamental sequence]\label{setup: fund:sq}
 Let $K$ be a compact subset of a finite dimensional $C^\omega_\K$ manifold $M$. 
 Then there always exists a sequence $U_1 \supseteq U_2 \cdots$ of metrisable open neighbourhoods of $K$ in $M$ such that
 \begin{itemize}
  \item [(a)] For each $n\in\N$, the set $\overline{U_{n+1}}$ is compact in $U_n$.
  \item [(b)] Each open neighbourhood $U$ of $K$ contains one of the sets $U_n, n\in \N$. 
  \item [(c)] Each connected component of each set $U_n$ intersects the compact set $K$ non trivially.
 \end{itemize}
 Such a sequence will be called a \emph{fundamental sequence} of open neighbourhoods of $K$ in $M$.
\end{setup}
\begin{proof}
 Note that in general, $M$ need not be metrisable. 
 However, $M$ is locally compact. Thus, the compact set $K$ has a relatively compact neighbourhood $U$. 
 The closure $\overline U$ is compact and locally metrisable, hence metrisable\footnote{For locally metrisable spaces, metrisability and paracompactness are equivalent.}. 
 Therefore, the compact set $K$ is contained in an open relatively compact and metrisable set $U$. 

 Using a metric on $U$, we construct a descending sequence $U_1 \supseteq U_2 \supseteq \ldots$ of open neighbourhoods of $K$ in $U$ such that every neighbourhood of $K$ in $M$ contains some $U_n$.  
 We can and will always choose a fundamental sequence such that every connected component of each $U_n$ meets $K$ and $\overline{U_{n+1}} \subseteq U_n$ for all $n$.
 Furthermore, the closure $\overline{U_{n+1}}$ is contained in $\overline{U}$ for all $n \in \N$ whence it is compact (in $M$ and also in $U_n$). 
\end{proof}

\begin{setup}[Germs of analytic mappings]										\label{setup_space_of_germs}
 Let $K$ be a compact subset of a finite dimensional $C^\omega_\K$ manifold $M$ and let $E$ be a locally convex vector space over $\K$.
 \begin{itemize}
  \item [(a)]
			 Let $C^\omega_\K (K \subseteq M, E)$ be the space of
			 \emph{germs of $\K$-analytic maps along $K$} of $C^\omega_\K$-functions from open $K$-neighbourhoods in $M$ to $E$.
			 
			 Again we set $\Hol (K \subseteq M, E) \coloneq C^\omega_\C (K \subseteq M, E)$.
 			 By abuse of notation the germ of $f$ around $K$ will be denoted as '$f$'.
 	\item [(b)] 
 	     Let $\K=\C$. 
 	     Consider the directed set $(\nN, \subseteq)$ of open neighbourhoods of $K$ in $M$ (partially) ordered by inclusion. 
  		 Then for $U,W \in \nN$ with $W \subseteq U$ we obtain continuous linear maps 
  		 \[
  		  \res_W^U \colon \Hol (U,E) \rightarrow \Hol (W,E), f \mapsto f|_W,
  		 \]
  		 which form bonding maps for an inductive system in the category of complex locally convex spaces. 
  		 Passing to the limit of the system, we obtain limit maps  
	     \begin{displaymath}
     		\res_K^W \colon \Hol (W, E) \rightarrow \Hol (K \subseteq M, E) , \quad W \in \nN,
    	 \end{displaymath}
  		 assigning to each function $f\in\Hol(W,E)$ the associated germ around $K$.
  		 We give $\Hol (K \subseteq M,E)$ the (a priori not necessarily Hausdorff)
  		 inductive limit topology of the above inductive system. 
  		 In Lemma \ref{lem_space_of_germs_is_Hausdorff} we will see that this topology is indeed Hausdorff. 
  \item [(c)]
  		 Let again, $\K=\C$ and fix a fundamental sequence $U_1\supseteq U_2 \supseteq \cdots$ of $K$ in $M$.
  		 By \ref{setup: fund:sq} (c) and the Identity Theorem for analytic mappings the bonding maps $\res^{U_n}_{U_m}$ are injective for $m \geq n$ and so are the limit maps.
  		 Now \ref{setup: fund:sq} (b) implies that the direct limit topology on $\Hol (K \subseteq M, E)$ discussed in part (b)
  		 equals the direct limit topology of the sequence
  		 $(\Hol (U_n,E), \res_{U_{n+1}}^{U_n})_{n\in \N}$.
  		 
	\textbf{For $E$ finite dimensional,} the topology on $\Hol (K \subseteq M,E)$ is nicer.
  		By \ref{setup: fund:sq} (a) The bonding maps of the inductive limit factor in the obvious way through 
    	 \begin{displaymath}
     			\Hol (U_n, E) \rightarrow \BHol (U_{n+1} , E) \rightarrow \Hol (U_{n+1},E).
    	 \end{displaymath}
  		 We conclude that $\Hol (K \subseteq M, E) = \displaystyle\lim_{\rightarrow} \BHol (U_n,E)$.
 
  		 The topology on $\Hol (K \subseteq M, \C)$ coincides with the inductive topology induced by the system 
       \[
         \BHol (U_n, \C) \rightarrow \Hol (K \subseteq M, \C), g \mapsto \res^{U_n}_K (g) \text{ for }n \in \N.
       \]
  		 In \cite[Theorem 3.4]{KM90} it was proved that the bonding maps of this system are compact,
  		 whence $\Hol (K \subseteq M, \C)$ becomes a Silva space\footnote{Recall that a Silva space is defined as the inductive limit of a sequence of Banach spaces such that the bonding maps are compact.}. 
		 For $k \in \N$ recall from \cite[Lemma 3.4]{hg2002} that $\Hol (K\subseteq M ,\C^k) \cong \Hol (K \subseteq M ,\C)^k$ as topological vector spaces.
  		 Thus $\Hol (K \subseteq M, \C^k)$ is a Silva space for all $k \in \N$ as a finite locally convex sums of Silva spaces. 
 \end{itemize}
\end{setup}

\begin{lem}[$\Hol(K\subseteq M,E)$ is Hausdorff]															\label{lem_space_of_germs_is_Hausdorff}
 Let $K$ be a compact subset of a finite dimensional $C^\omega_\C$ manifold $M$ and let $E$ be a complex locally convex vector space. 
 Then $\Hol(K\subseteq M,E)$ with the topology of \ref{setup_space_of_germs}(b) is Hausdorff.
\end{lem}
\begin{proof}
  For each $a\in M$, let $\V_a$ denote the set of all charts around $a$, i.e.\ the set $\V_a$ consists of all $C^\infty_\C$-diffeomorphisms $\smfunc{\phi}{U_\phi}{V_\phi}$ with $U_\phi$ open $a$-neighbourhood in $M$ and
  $V_\phi$ open in $\C^d$ with $d = \dim M$.
  Denote by $\nN$ the family of all open $K$-neighborhoods in $M$. For each $W\in\nN$, Lemma \ref{lem: cocinfty} yields a continuous linear map 
  \[
   \func{\Psi_W}{\Hol(W,E)}{\prod_{a\in K}\prod_{\phi\in\V_a} \prod_{\alpha\in\N_0^d} E }{\gamma}{\left( \partial^\alpha( \gamma\circ\phi^{-1} ) (\phi(a) ) \right)  }.
  \]
  By construction, for $W,U \in \nN$ with $W \subseteq U$ we have $\Psi_W \circ \res_W^U = \Psi_U$. 
  Hence on the locally convex limit $\Hol(K\subseteq M,E)$ these maps induces a continuous linear map 
  \[
   \smfunc{\Psi}{\Hol(K\subseteq M,E)}{\prod_{a\in K}\prod_{\phi\in\V_a} \prod_{\alpha\in\N_0^d} E. }
  \]
  This map is injective by the Identity Theorem for complex analytic maps. Hence, we obtain an injective continuous map from $\Hol(K\subseteq M,E)$ into a Hausdorff space, implying the Hausdorff property of the space of germs. 
\end{proof}

We will now study sections of locally convex vector bundles.
Our goal is now to topologise the spaces of germs of analytic sections around compact subsets.

We start with a small lemma about infinite dimensional vector bundles:

 \begin{lem}[Regularity of total spaces of bundles] \label{lem: reg:bun}
  Let $(F,\pi,M)$ be a topological vector bundle, i.e.\ a vector bundle whose typical fibre $E$ is a topological convex space. 
  The total space $F$ is regular as a topological space if and only if the topological space $M$ is regular.
 \end{lem}
 \begin{proof}
  The space $M$ can be embedded in the total space $F$ via the zero-section. Hence, $M$ is regular if $F$ is regular. To show the converse implication, we assume that $M$ is regular.
  
  Recall from general topology that $F$ is regular if every open neighbourhood of every point contains a closed neighbourhood of that point. 
  To check this criterion fix $a\in F$ together with an open $a$-neighbourhood $\Omega\subseteq F$.
  
  Choose a trivialisation $\kappa \colon \pi^{-1} (M_\kappa) \rightarrow M_\kappa \times E$ of the bundle with $a \in \pi^{-1} (M_\kappa)$.
  The set $\kappa(\Omega \cap \pi^{-1}(M_\kappa))$ is an open neighbourhood of $\kappa(a)$ in the product $M\times E$. 
  Hence, we can find open subsets $W \subseteq M_\kappa$ and $V \subseteq E$, respectively, such that
  \[
   \kappa (a) \in  W \times V \subseteq \kappa(\Omega \cap \pi^{-1}(M_\kappa)).
  \]
  Now $M_\kappa$ is regular as a subspace of the regular space $M$ and $E$ is regular as a topological vector space. 
  Therefore, we may chose $W$ and $V$ so small that $\overline W \times \overline V \subseteq \kappa(\Omega \cap \pi^{-1}(M_\kappa))$.
  We obtain an $a$-neighbourhood 
  $
   A\coloneq \kappa^{-1}(\overline W \times \overline V) \subseteq F
  $
  contained in $\Omega$. 
  It remains to show that $A$ is closed in $F$.
  
  Consider an element $\overline a \in \overline A$ of the closure of $A$. We apply the bundle projection $\pi$ to this element and obtain $\pi(\overline a) \in \pi (\overline A ) \subseteq \overline{\pi(A)} = \overline{ \overline W } =\overline W\subseteq M_\kappa$.
  As $\overline a \in \pi^{-1}(M_\kappa)$ we apply $\kappa$ to achieve 
  $
   \kappa (\overline a ) \in \kappa (\overline A) \subseteq \overline{\kappa(A)} = \overline{ \overline W \times \overline V }= \overline W \times \overline V . 
  $
  This shows that $\overline a \in \kappa^{-1} (\overline W \times \overline V) =A$ which shows that $A$ is closed and concludes the proof.
   \end{proof}

\begin{defn}\label{defn: ngermtop}
 Let $(F,\pi,M)$ be a $C^\omega_\C$-bundle whose typical fibre $E$ is a complex locally convex vector space and $M$ is finite dimensional. 
 \begin{enumerate}
  \item Let $\Gamma_\C^\omega (F)$ be the \emph{space of holomorphic sections} of $(F,\pi,M)$. We topologise $\Gamma_\C^\omega (F)$ with the initial topology with respect to the maps 
    \begin{displaymath}
     \theta_\psi \colon \Gamma_\C^\omega (F) \rightarrow \Hol(M_\psi, E) , X \mapsto \text{pr}_2 \circ \psi \circ X|_{M_\psi} .
    \end{displaymath}
  Here $\psi$ ranges through all bundle trivialisations of $F$. 
  Note that $\Gamma_\C^\omega (F)$ is Hausdorff as the point evaluations are continuous. 
 \end{enumerate}
 Consider a compact subset $K \subseteq M$. 
\begin{enumerate}
  \item[(b)] Let $\nN_K$ be the \emph{set of all open $K$-neighbourhoods}.
  For $U \in \nN_K$ we define the \emph{restricted bundle} $(F|U\coloneq \pi^{-1} (U), \pi|_{F|U}^U,U)$.
  \item[(c)] We denote by $\Gamma^\omega_\C (F|K)$ the \emph{space of germs of sections along $K$}, i.e.\ germs of sections in $\Gamma_\C^\omega(F|U)$ where $U$ ranges $\nN_K$. 
  
  Topologise $\Gamma^\omega_\C (F|K)$ as the the locally convex inductive limit of the cone \\
  $(\Gamma^\omega_\C (F|W) \rightarrow \Gamma^\omega_\C (F|K))_{W \in \nN_K}$ (where the limit maps send a section to its germ). 
 \end{enumerate}
 At this point it is not clear whether $\Gamma^\omega_\C (F|K)$ is Hausdorff.
 We will see in Lemma \ref{lem: topchar} that the spaces $\Gamma^\omega_\C (F|K)$ are also Hausdorff.
\end{defn}

\begin{lem}\label{lem: cotop}
 Let $(F,\pi,M)$ be a $C^\omega_\C$-bundle whose typical fibre $E$ is a complex locally convex space and $M$ is finite dimensional.  
 The topology of $\Gamma^\omega_\C (F)$ (cf.\ Definition \ref{defn: ngermtop}) coincides with the compact open topology. 
 Hence a typical subbasis for the topology is 
 \begin{displaymath}
  \lfloor L,O \rfloor \coloneq \setm{X \in \Gamma^\omega_\C (F)}{X(L)\subseteq O}
 \end{displaymath}
 where $L$ is a compact subset of $M$ and $O$ is an open subset of $F$.
\end{lem}

\begin{proof}
 We show first that for a compact set $L \subseteq M$ and open subset $O \subseteq F$ the set $\lfloor L,O\rfloor$ is open in $\Gamma^\omega_\C (F)$. 
 To this end, we will prove that $\lfloor L,O\rfloor$ is a neighbourhood for each fixed element $\sigma \in \lfloor L,O\rfloor$.
  Since $L \subseteq M$ is compact and $M$ is finite dimensional, there is a finite family of compact sets $K_\alpha \subseteq M,\ 1\leq \alpha \leq m$ such that: 
  \begin{enumerate}
   \item $L = \bigcup_{\alpha} K_\alpha$,
   \item for each $\alpha$ there is a bundle trivialisation $\psi_\alpha$ with $K_\alpha \subseteq M_{\psi_\alpha}$ and 
   \item the compact set $K_\alpha \times \pr_2 \circ \psi_\alpha \circ \sigma (K_\alpha)$ is contained in $M_{\psi_\alpha} \times \pr_2 \circ \psi_\alpha (O \cap \dom \psi_\alpha)$.
  \end{enumerate}
 Recall that the topology on $\Gamma^\omega_\C (F)$ is initial with respect to maps $\theta_\psi$ where $\psi$ runs through all bundle trivialisations. 
 Moreover, the space $\theta_\psi$ maps into carries the compact open topology. 
 As we are dealing with sections the following identity holds
  \begin{displaymath}
   \Omega_\alpha \coloneq \lfloor K_\alpha , (\pr_2 \circ \psi_\alpha)^{-1} (\pr_2 \circ \psi_\alpha (O \cap \dom \psi_\alpha))\rfloor = \theta_{\psi_\alpha}^{-1} (\lfloor K_\alpha , \pr_2 \circ \psi_\alpha (O\cap \dom \psi_\alpha)\rfloor.
  \end{displaymath}
 In particular we observe that $\sigma$ is contained in each open set $\Omega_\alpha$ for $1\leq \alpha \leq m$.
 Note that by construction the set $\bigcap_{1\leq \alpha \leq m} \Omega_\alpha$ is contained in $\lfloor L, O\rfloor$ which proves that $\lfloor L,O\rfloor$ is a neighbourhood of $\sigma$.
 
 Conversely, fix a section $\tau$ together with an arbitrary $\tau$-neighbourhood $\Omega$ in $\Gamma^\omega_\C (F)$.
 By definition of the initial topology, $\Omega$ contains an open $\tau$-neighbourhood of the form 
  \begin{displaymath}
   \bigcap_{1\leq k \leq n} \theta_{\psi_k}^{-1} (\lfloor L_k, U_k\rfloor), \text{ where } L_k \subseteq M_{\psi_k} \text{ is compact and } U_k \subseteq E \text{ is open}. 
  \end{displaymath}
 As above $\theta_{\psi_k}^{-1} (\lfloor L_k , U_k\rfloor) = \lfloor L_k , (\pr_2 \circ \psi_k)^{-1} (U_k)\rfloor$. Thus the assertion follows.
\end{proof}

\begin{lem}\label{lem: restro}
 Let $(F,\pi,M)$ be a $C^\omega_\C$-bundle whose typical fibre $E$ is a complex locally convex space and $M$ is finite dimensional. 
 Consider a compact subset $K$ of $M$.
  \begin{enumerate}
   \item If there is a trivialisation $\psi \colon F \supseteq \Omega \rightarrow M_\psi \times E$ such that $K \subseteq M_\psi$, 
   then the map $I_\psi \colon \Gamma^\omega_\C (F|K) \rightarrow C^\omega_\C (K\subseteq M_\psi, E) , \gamma \mapsto \pr_2 \circ \psi \circ \gamma$
   is an isomorphism of locally convex spaces.  
   \item If $L \subseteq K$ is another compact subset, then the canonical restriction map\\
   $\res_{L}^K \colon \Gamma^\omega_\C (F|K) \rightarrow \Gamma^\omega_\C (F|L)$ is continuous linear.
   \item Let $K_1 , K_2 \subseteq M$ be compact subsets with $K = K_1 \cup K_2$. Then the map  
	      \begin{displaymath}
	       R \coloneq (\res_{K_1}^K , \res_{K_2}^K) \colon \Gamma^\omega_\C (F|K) \rightarrow \Gamma^\omega_\C (F|K_1) \oplus \Gamma^\omega_\C (F|K_2) 
	      \end{displaymath}
        is a topological embedding. In particular if $K_1 \cap K_2 = \emptyset$ the mapping $R$ is an isomorphism of locally convex spaces.
   \end{enumerate}
 \end{lem}

 \begin{proof}
  \begin{enumerate}
   \item Let $W \subseteq M_\psi$ be an open neighbourhood of $K$. 
    From Lemma \ref{lem: cotop} we deduce that the linear bijective map $(\pr_2 \circ \psi)_* \colon \Gamma^\omega_\C (F|W) \rightarrow \Hol (W, E), \sigma \mapsto \pr_2 \circ \psi \circ \sigma$ is an isomorphism of topological vector spaces.
    Now the assertion follows from an easy inductive limit argument.
   \item Since the restriction map $\res^W_U \colon \Hol (W,E) \rightarrow \Hol (U,E)$ is continuous linear for all open $U \subseteq W$, the assertion follows from an inductive limit argument.
   \item By (b) $R$ is continuous linear. Clearly $R$ is also injective.
   If $K_1$ and $K_2$ are disjoint, the assertion is trivial.
   
   Thus we will now assume that $K_1 \cap K_2 \neq \emptyset$.
   Fix fundamental sequences of neighbourhoods $(U_n^i)_{n \in \N}$ of $K_i$ for $i \in \{1,2\}$.
   By construction the sets $U_n \coloneq U_n^1 \cup U_n^2$ form a fundamental sequence for $K$.
  
   It remains to show that the map $R$ is a topological embedding.
   To this end let $\Omega$ be an open zero-neighbourhood in $\Gamma^\omega_\C (F|K)$. It remains to show that $R(\Omega)$ is an open zero-neighbourhood in $R(\Gamma^\omega_\C (F|K))$.
   By \cite[p.109]{wengenroth} we may assume that $\Omega$ is a zero-neighbourhood of the form
   \begin{displaymath}
    \Omega = \bigcup_{n \in \N} \sum_{1\leq j \leq n} \Omega_j \text{ with } \Omega_j \subseteq \Gamma^\omega_\C (F|U_j) \text{ open zero-neighbourhood}
   \end{displaymath}
   By Lemma \ref{lem: cotop} we may assume that $\Omega_j = \lfloor L_j , O_j\rfloor$ with $L_j \subseteq U_j$ compact and $O_j \subseteq F$ open.
   Hence every $\Omega_j$ gives rise to two open sets $\lfloor L_j \cap \overline{U_{j+1}^i}, O_j\rfloor \subseteq \Gamma^\omega_\C (F|U_j^i)$ for $i \in \{1.2\}$.
   As we deal with sections one easily obtains the equality 
   \begin{displaymath}
    R(\Omega) = \bigcup_{n\in \N} \sum_{1\leq j \leq n} \left(\lfloor L_j \cap \overline{U_{j+1}^1}, O_j\rfloor \times \lfloor L_j \cap \overline{U_{j+1}^2}, O_j\rfloor\right) \cap R(\Gamma^\omega_\C (F|K))
   \end{displaymath}
   As $\Gamma^\omega_\C (F|K_1) \oplus \Gamma^\omega_\C (F|K_1)$ is the inductive limit of the system $\Gamma^\omega_\C (F|U_n^1) \oplus \Gamma^\omega_\C (F|U_n^2)$, we see that $R(\Omega)$ is open in the subspace topology.
   Summing up $R$ is a topological embedding.   \qedhere
  \end{enumerate}
\end{proof}

\begin{lem}\label{lem: topchar}
 Let $(F,\pi,M)$ be a $C^\omega_\C$-bundle whose typical fibre $E$ is a complex locally convex space and $M$ is finite dimensional.
 Fix a compact subset $K \subseteq M$ and a finite family of compact subsets $(K_\alpha)_{\alpha \in A}$ of $M$ such that 
  \begin{enumerate}
   \item $K = \bigcup_\alpha K_\alpha$
   \item For each $\alpha$ there is a bundle trivialisation $\psi_\alpha$ with $K_\alpha \subseteq M_{\psi_\alpha}$. 
  \end{enumerate}
 With the notation of Lemma \ref{lem: restro} we obtain a mapping  
  \begin{displaymath}
   \Theta \coloneq (I_{\psi_\alpha} \circ \res^K_{K_\alpha})_{\alpha \in A} \colon \Gamma^\omega_\C(F|K) \rightarrow \bigoplus_{\alpha \in A} \Hol (K_\alpha \subseteq M_{\psi_\alpha} , E) 
  \end{displaymath}
 is a linear topological embedding, whose image is a closed vector subspace.
 Thus $\Gamma^\omega_\C (F|K)$ is Hausdorff. If $E$ is finite dimensional, the space $\Gamma^\omega_\C(F|K)$ is a Silva space.
\end{lem}

\begin{proof} Iteratively applying Lemma \ref{lem: restro} (c) we obtain a linear topological embedding 
  \begin{displaymath}\
   R_A = (\res^K_{K_\alpha})_{\alpha \in A} \colon \Gamma^\omega_\C (F|K) \rightarrow \bigoplus_{\alpha \in A} \Gamma^\omega_\C (F|K_\alpha).
  \end{displaymath}
 Apply Lemma \ref{lem: restro} (a) to each summand to see that $\Theta$ is a topological embedding.
 Now $\Gamma^\omega_\C (F|K)$ embeds into a product of Hausdorff spaces (cf.\ Lemma \ref{lem_space_of_germs_is_Hausdorff}) and thus $\Gamma^\omega_\C (F|K)$ is Hausdorff.
 Moreover, the image $\im \Theta$ is homeomorphic to $\im R_A \subseteq \bigoplus_{\alpha \in A} \Gamma^\omega_\C (F|K_\alpha)$.
 
  It is easy to see that the image of $R_A$ is the space 
  \begin{displaymath}
   \setm{(\gamma_\alpha)_{\alpha \in A} \in \bigoplus_{\alpha \in A} \Gamma^\omega_\C (F|K_\alpha)}{\res^{K_\alpha}_{K_\alpha \cap K_\beta} (\gamma_\alpha) =  \res^{K_\beta}_{K_\alpha \cap K_\beta} (\gamma_\beta), \text{ if } K_\alpha \cap K_\beta \neq \emptyset  }.
  \end{displaymath}
 As each $\Gamma^\omega_\C (F|K_\alpha \cap K_\beta)$ is Hausdorff the image of $R_A$ is obviously closed.
 
 Let $E$ now be finite dimensional. Then each $\Hol (K_\alpha \subseteq M_{\psi_\alpha} ,E)$ is a Silva space by \ref{setup_space_of_germs}(c).
 Since any finite locally convex sum of Silva spaces and closed subspaces of Silva spaces are Silva spaces by \cite[Corollary 8.6.9]{barr87}, the assertion follows. 
 \end{proof}

\phantomsection
\addcontentsline{toc}{section}{References}
\bibliographystyle{new}
\bibliography{analytic}
\end{appendix}
\end{document}